\documentclass[3p]{elsarticle}
\usepackage{amsmath}
\usepackage{amssymb}
\usepackage{hyperref}

\newtheorem{theorem}{Theorem}
\newtheorem{claim}{Claim}
\newtheorem{corollary}{Corollary}
\newdefinition{definition}{Definition}
\newproof{example}{\textbf{Example}}
\newtheorem{lemma}{Lemma}
\newdefinition{notation}{{\bf Notation}}
\newtheorem{proposition}{Proposition}
\newdefinition{remark}{Remark}
\newproof{proof}{{\bf Proof}}

\numberwithin{equation}{section}

\newcommand{\I}{\infty}
\renewcommand{\S}{\sigma}

\newcommand{\T}{\tilde}

\newcommand{\Z}{\zeta}
\renewcommand{\L}{\lambda}
\newcommand{\Gm}{\varGamma}
\newcommand{\LL}{\varLambda}
\renewcommand{\r}{\varrho}

\renewcommand{\P}{\partial}
\newcommand{\B}{\beta}
\newcommand{\D}{\delta}
\newcommand{\Lap}{\triangle}
\renewcommand{\O}{\overline}
\newcommand{\A}{\alpha}
\newcommand{\lng}{\langle}
\newcommand{\rng}{\rangle}
\newcommand{\omg}{\omega}
\newcommand{\V}{\varepsilon}
\newcommand{\vp}{\varphi}
\renewcommand{\t}{\tau}
\renewcommand{\th}{\theta}

\newcommand{\wh}{\widehat}
\newcommand{\QED}{\qquad\mbox{\qed}}
\renewcommand{\wh}{\widehat}
\newcommand{\U}{\underline}
\newcommand{\VS}{\varSigma}
\newcommand{\VP}{\varPi}
\newcommand{\Vt}{\vartheta}

\newcommand{\R}{\mathbb{R}}
\newcommand{\BN}{\mathbb{N}}

\newcommand{\CB}{\mathcal{B}}
\newcommand{\CG}{\mathcal{G}}
\newcommand{\CE}{\mathcal{E}}
\newcommand{\CF}{\mathcal{F}}

\newcommand{\CI}{\mathcal{I}}

\newcommand{\CK}{\mathcal{K}}

\newcommand{\CS}{\mathcal{S}}

\newcommand{\CM}{\mathcal{M}}
\newcommand{\CN}{\mathcal{N}}

\renewcommand{\CG}{\mathcal{G}}

\newcommand{\CX}{\mathcal{X}}
\newcommand{\CZ}{\mathcal{Z}}
\newcommand{\Uv}{\underline{\vartheta}}
\newcommand{\vU}{\varUpsilon}
\newcommand{\Ov}{\overline{\vartheta}}
\renewcommand{\leq}{\leqslant}
\renewcommand{\geq}{\geqslant}

\DeclareMathOperator{\sign}{sign}

\journal{Journal of Differential Equations}

\begin{document}

\begin{frontmatter}

\title{Global well-posedness and grow-up rate of solutions for a sublinear pseudoparabolic equation 
}

\author[chu]{Sujin Khomrutai\corref{cor1}}
\ead{sujin.k@chula.ac.th}

\address[chu]{Department of Mathematics and Computer Science, Faculty of Science, Chulalongkorn University, Bangkok 10330, Thailand}

\cortext[cor1]{Corresponding author.}


\begin{abstract}
We study positive solutions of the pseudoparabolic equation with a sublinear source in $\mathbb{R}^n$. In this work, the source coefficient could be unbounded and time-dependent. Global existence of solutions to the Cauchy problem is established within weighted continuous spaces by approximation and monotonicity arguments. Every solution with non-zero initial value is shown to exhibit a certain lower grow-up and radial growth bound, depending only upon the sublinearity and the unbounded, time-dependent potential. Using the lower grow-up/growth bound, we can prove the key comparison principle. Then we settle the uniqueness of solutions for the problem with non-zero initial condition by employing the comparison principle. For the problem with the zero initial condition, we can classify the non-trivial solutions in terms of the maximal solutions. When the initial condition has a power radial growth, we can derive the precise asymptotic grow-up rate of solutions and obtain the critical growth exponent.
\end{abstract}

\end{frontmatter}

\section{Introduction}

We study the existence, uniqueness or non-uniqueness, and grow-up rate of solutions $u=u(x,t)\geq0$ to the semi-linear pseudoparabolic Cauchy problem 
\begin{align}\label{Eqn:Main}
\begin{cases}
\P_tu-\Lap\P_tu=\Lap u+V(x,t)u^p&x\in\R^n,t>0,\\
\qquad u(x,0)=u_0(x)&x\in\R^n,
\end{cases}
\end{align}
where $0<p<1$ is a constant, $V,u_0\geq0$ are given continuous functions, and $n\geq1$. In this work, the potential function $V$ can be non-autonomous (that is time-dependent) and unbounded. Our framework can be applied to a more general equation of the form
\begin{align}
\P_tu-\nu\Lap\P_tu=\Lap u+S(x,t,u)
\end{align}
where $\nu\geq0$ is a constant and the source function $S$ satisfies $0\leq S(x,t,u)\lesssim V(x,t)|u|^p$. This equation, of course, includes the sublinear heat equations with unbounded non-autonomous source.\smallskip

Pseudoparabolic equations are models of many important nonlinear physical systems (see for instance \cite{AKS11,BZK60,BBM72,KS03,Padron04} and the references therein). Also, in recent years, there has been a great interest in studying nonlinear evolution equations with unbounded, or singular, and even time-dependent coefficients.\smallskip

When the viscosity term $\Lap\P_tu$ is dropped, Eq.\ (\ref{Eqn:Main}) becomes the heat equation which is closely related with the pseudoparabolic equation \cite{RaoTing72,ShT70,Ting63,Ting69}. Nonlinear heat equations and systems on a (bounded or unbounded) domain have been studied extensively and are quite well understood. However, the nonlinear pseudoparabolic equations, especially those considered on an unbounded domain, have received very few attentions. This could be explained from the difference of their Green kernels:
\begin{align}
H(x,t)=\CF^{-1}\left(e^{-t|\xi|^2}\right),\qquad G(x,t)=\CF^{-1}\left(e^{-t|\xi|^2/(1+|\xi|^2)}\right)
\end{align}
for the linear heat and pseudoparabolic equations, respectively. We note that if $u\in\CS'$, a tempered distribution, then $H\ast u\in\CS$, however, we only have $G\ast u\in\CS'$ (\cite{JoshiFried98}). In addition, $H\simeq t^{-n/2}e^{-|x|^2/(4t)}$ but there is no an explicit expression for the pseudoparabolic kernel. In fact, the kernel $G$ is very complicated, see Eq.\ \ref{Eqn:Green}.\smallskip

To our knowledge, the first study on positive solutions of semilinear pseudoparabolic equations on $\R^n$ is the work by Cao et al.\ \cite{CaoYinWang09} (see also \cite{KaiNaumShish05} for an inspiring work on real value solutions). They considered Eq.\ (\ref{Eqn:Main}) with a constant potential, that is
\begin{align}
\P_tu-\Lap\P_tu=\Lap u+u^p\qquad\mbox{in $\R^n\times(0,\I)$},
\end{align}
for any constant $p>0$. In the sublinear case, that is $0<p<1$, they established the existence of solutions within $BC(\R^n)\triangleq C(\R^n)\cap L^\I(\R^n)$ by the method of sub- and super-solutions, but left the uniqueness an open problem. Recently, this problem was settled in \cite{Khomrutai15}. It was found that the uniqueness is guaranteed provided that the initial condition is non-trivial, whereas all non-trivial solutions with the zero initial condition are shown to be the delays of the maximal solution $u_\ast=((1-p)t)^{1/(1-p)}$. It is remarkable that these are the same findings as that for the sublinear heat equation \cite{AguirEscobedo86}.\smallskip

Let us explain our main results. In this work, we extend the results of \cite{CaoYinWang09} and \cite{Khomrutai15}. The solutions to Eq.\ (\ref{Eqn:Main}) are studied in the non-local (or mild) formulation
\begin{align}\label{Eqn:MildOrig}
u(x,t)=\CG(t)u_0+\int_0^t\CG(t-\t)\CB\left(Vu^p\right)(x,\t)d\t,
\end{align}
where $\CB\triangleq (1-\Lap)^{-1}$ is the Bessel potential operator and $\CG(t)\triangleq e^{-t\CB\Lap}$ is the pseudoparabolic Green operator (see (\ref{Eqn:Green}). A part of this work is inspired by the work of J.\ Aguirre and M.A.\ Escobedo \cite{AguirEscobedo86} on the sublinear heat equation $\P_tu=\Lap u+u^p$, $0<p<1$. The techniques in that paper, like most other studies on nonlinear heat equations, rely heavily on the explicit formula of the heat kernel and the scaling (or self-similarity) property of the heat equaton. For our work on the pseudoparabolic equation, however, we cannot follow their techniques because of the complicated green kernel and that (\ref{Eqn:Main}) admits no obvious scaling transformations. Our study reveals that such difficulty can be overcome and what we really need are some qualitative properties (Lemma \ref{Lem:BasicBG}, Proposition \ref{Prop:BoundBG}) and asymptotic behaviors (Lemma \ref{Lem:B1}) of $\CB$ and $G$.  Additional difficulty arises from the unboundedness and time-varying of $V$ in the source term, which was not included in \cite{AguirEscobedo86}. We overcome this difficulty by analysing the problem in the continuous spaces 
\begin{align}
BC_a=BC_a(\R^n)\triangleq \left\{\vp\in C(\R^n):|\vp(x)|\leq\mbox{const}\cdot|x|^{a}\,\,\mbox{as $|x|\to\I$}\right\},
\end{align}
which is endowed with a weighted norm 
\begin{align}
\|\vp\|_{R,a}\triangleq \left\|\left(R^2+|\cdot|^2\right)^{-\frac{a}{2}}\vp\right\|_{L^\I},
\end{align}
where $R>0$ and $a\geq0$. By choosing $R$ sufficiently large, we can establish the existence result within $BC_a$. The reason for using the norm $\|\cdot\|_{R,a}$ instead of the usual $\|\cdot\|_{1,a}$ as in \cite{AguirEscobedo86} is that the operators $\CB$ and $\CG(t)$, as $t\to0^+$, behave `almost' as contractions on $BC_a$ provided $R$ is sufficiently large. See Proposition \ref{Prop:BoundBG} (ii). This property and the corresponding operator norm estimates play a crucial role throughout our study. \smallskip

Some important preliminaries are given in Section \ref{Sec:Prelim}. The global existence result of Eq.\ (\ref{Eqn:Main}) is proved in Section \ref{Sec:Glob}. We introduce a new dependent variable $w$ such that $w(\cdot,t)\in C_b(\R^n)=BC_0$ and study some approximate problems where the nonlinear term $w^p$, which is not Lipschitz continuous at $w=0$, is replaced by Lipschitz continuous functionals $F(w)$. The global existence and comparison principle for the approximate problems can be proved by standard contraction mapping and Gronwall type arguments. Then we prove the existence result of Eq.\ (\ref{Eqn:Main}) using the monotonicity property of the approximated solutions. We obtain the global existence result (Theorem \ref{Thm:Existence}) under the assumption that there is a continuous function $\L(t)\geq0$ such that $V$ satisfies
\begin{align}\label{HypVU}\tag{H1}
V(x,t)\lesssim\L(t)|x|^\S\quad\mbox{as $|x|\to\I$ and $t\to\I$, where $\S\in\R$},
\end{align}
and the solutions belong to $BC_a$ with
\begin{align}\label{HypAL}\tag{H2}
a\geq\frac{\S_+}{1-p}.
\end{align}
All the results in this work are applicable to many important potentials such as
\[
V\sim t^k|x|^\S,\quad(\log t)^\nu|x|^\S,\quad t^k(\log t)^\nu|x|^\S,
\]
as $|x|\to\I,t\to\I$, where $k,\nu,\S\in\R$. Observe that the sublinearity of the source term always allows the solutions to exist globally in time, no matter how the potential $V$ behaves at infinity and regardless of the initial condition $u_0$. This is in contrast to the superlinear case $p>1$, where it was shown in \cite{KhomrutaiNA15} that $V$ can induce the blowing up phenomena if $1<p\leq1+\frac{\S+2}{n}$ or $u_0$ is sufficiently large. The same phenomena occurred in the heat equations.\smallskip

We establish a lower bound for the grow-up (in time) and spatial (radial) growth of solutions in Section \ref{Sec:Grow}. The solutions of (\ref{Eqn:Main}) will be expressed in both (\ref{Eqn:MildOrig}) and a modified formulation
\begin{align}\label{Eqn:MildMod}
\mu(x,t)=u_0+\int_0^t\CB\mu(x,\t)d\t+\int_0^te^{(1-p)\t}\CB\left(V\mu^p\right)(x,\t)d\t,
\end{align}
which is obtained by formally setting 
\begin{align}
u=e^{-t}\mu
\end{align}
into (\ref{Eqn:Main}). We prove in Lemma \ref{Lem:TMisM} and Corollary \ref{Cor:PimM} that if $\mu$ solves (\ref{Eqn:MildMod}) then $u$ solves (\ref{Eqn:MildOrig}). By dropping the nonlinear term, denoted $\CN^s\mu$ (see also (\ref{NonLins})), in (\ref{Eqn:MildMod}) we can prove a preliminary estimate that $\mu(t_0)\geq C_0e^{-\D|x|}>0$, where $\D>1$ and $t_0>0$, provided $u_0\neq0$. Applying this lower estimate into $\mu\geq\CN^s\mu$, which is true by (\ref{Eqn:MildMod}), and performing iteration, we can derive the lower grow-up/growth bound for the solutions in Theorem  \ref{Thm:uniflower}. The main assumption for this result is that there is a continuous function $\L(t)>0$ such that $V$ satisfies
\begin{align}\label{HypVL}
V(x,t)\gtrsim\L(t)|x|^\S\quad\mbox{as $|x|\to\I$ for all $t>0$, where $\S\in J_n^+$},\tag{H3}
\end{align}
and 
\[
J_n^+\triangleq\{0\}\cup\left[(2-n)_+,\I\right)\quad(n\geq1);
\]
see (\ref{HypVLp}). An important consequence of this lower grow-up/growth bound is that when $V=\L(t)|x|^\S$, $\S>0$ and $\L(t)>0$ an arbitrary continuous function, the non-autonomous, unbounded sublinear source induces the growth on the solutions at least as $|x|^{\S/(1-p)}$, at any time $t>0$, even when the initial condition is decaying. It also induces a grow-up rate at least as $(\int_0^t\L(s)ds)^{1/(1-p)}t^{\S/(2(1-p))}$, see Theorem  \ref{Thm:uniflower} and Remark \ref{Rem:Grow}. This result is adding to the already observed phenomenon in \cite{AguirEscobedo86} and \cite{Khomrutai15} that the sublinear  source with constant potential, in the heat equation and the pseudoparabolic equation, induces the 
grow-up rate (in time) at least as $t^{1/(1-p)}$ on the solutions. 
\smallskip

Next, in Section \ref{Sec:CompU}, we prove the key comparison principle (Theorem \ref{Thm:Comp}) assuming that 
\begin{align}\tag{H4}\label{Hyp:Vcomp}
V(x,t)\sim\L(t)|x|^\S\quad\mbox{for all $x\in\R^n$, $t>0$, where $\L(t)>0,\S\in J_n^+$}
\end{align}
and $0<p<1$ is sufficiently small. More precisely, if $V$ satisfies (\ref{HypV}) where $0<\Uv\leq\Ov$ are constants, then the result of the theorem is true provided
\begin{align}\tag{H5}\label{HypCompOnp}
0<p<\V_0(\Ov/\U v),
\end{align}
where $\V_0>0$ is the constant given in (\ref{Est:BVd}); see Remark \ref{Rem:SharpComp} (i). In proving the comparison principle, we use the results in Proposition \ref{Prop:BoundBG} that $\CB$ and $\CG(t)$, as $t\to0^+$, behave almost as contractions and the lower estimate of the grow-up/growth of solutions with non-trivial initial condition. This comparison principle is optimal in view of Remark \ref{Rem:SharpComp} (ii). Some variations of the comparison theorem is also presented (Corollary \ref{Cor:Comp1}, Corollary \ref{Cor:Compare}). Using the comparison principle, the uniqueness of solutions (Theorem \ref{Thm:Unique}) for Eq.\ (\ref{Eqn:Main}) is established when the initial condition $u_0\neq0$. \smallskip

In Section \ref{Sec:Zero} we analyse Eq.\ (\ref{Eqn:Main}) when the initial condition $u_0=0$. That is we consider the problem
\begin{align}\label{Eqn:zero}
\begin{cases}
\P_tu-\Lap\P_tu=\Lap u+V(x,t)u^p&x\in\R^n,t>0,\\
\qquad\,\,\,u(\cdot,0)=0&x\in\R^n.
\end{cases}
\end{align}
We begin by introducing the notion of maximal solution, for a potential $V$, and establish its uniqueness (Theorem \ref{Thm:UniqueMax}). For each potential $V$, the maximal solution is denoted by $u_{\ast V}$. Under the assumption (\ref{HypNonUnq}), we can classify the non-trivial solutions of (\ref{Eqn:zero}) under various circumstances:
\begin{align*}
\begin{cases}
 u=\left((1-p)\int_0^{[t-\kappa]_+}\L(s+\kappa)ds\right)^{\frac{1}{1-p}}&\mbox{if $V=\L(t)\geq0$, a continuous function},\\
\vspace{-10pt}\\
\displaystyle u=u_{\ast V(x)}([t-\kappa]_+)&\mbox{if $V=V(x)$ satisfying (\ref{HypNonUnq})},\\
\vspace{-10pt}\\
\displaystyle u=u_{\ast V^\kappa}(x,[t-\kappa]_+)&\mbox{if $V=V(x,t)$ satisfying (\ref{Hyp:Convex})},
\end{cases}
\end{align*}
where $\kappa\geq0$ is a constant. See Theorem \ref{Thm:Classf1}, \ref{Thm:Classf2}, and \ref{Thm:Classf3}, respectively. It is remarkable that, in the first circumstance, we have obtained the explicit formula for all non-trivial solutions in terms of $\L(t)$. This is a generalization of the result obtained in (\cite{Khomrutai15}) when $\L=1$. When the potential $V$ is both $x$- and $t$-dependent, a convexity condition is required to be imposed on the potential. 
\smallskip

Finally, in Section \ref{Sec:Grow-up}, we study the asymptotic grow-up rate of solutions of (\ref{Eqn:Main}) under the assumption that $u_0$ has a radial growth $\sim|x|^a$ where $a\in\R$. More precisely, we assume that
\begin{align}\tag{H6}\label{HypSharp}
\liminf_{|x|\to\I}|x|^{-a}u_0(x)=l_1,\quad \limsup_{|x|\to\I}|x|^{-a}u_0(x)=l_2,
\end{align}
where $l_1,l_2\in(0,\I)$ are constants and $V$ satisfies. In Theorem \ref{Thm:a>ac}, \ref{Thm:a=a_c}, and \ref{Thm:a<ac}, we show that asymptotically as $t\to\I$,
\begin{align*}
\begin{cases}
\displaystyle l_1t^{\frac{a}{2}}\lesssim\|u(\cdot,t)\|_{R,a}\lesssim l_2t^{\frac{a}{2}}&\mbox{if $a>a_c$},\\
\vspace{-10pt}\\
\displaystyle l_1t^{\frac{a}{2}}\lesssim\|u(\cdot,t)\|_{R,a}\lesssim_\V l_2t^{\frac{a}{2}+\V}&\mbox{if $a=a_c$, for any $\V>0$},\\
\vspace{-10pt}\\
\displaystyle t^{\frac{a_c}{2}}\lesssim\|u(\cdot,t)\|_{R,a}\lesssim_\V l_2^\V t^{\frac{a_c}{2}+\V}&\mbox{if $a_0\leq a<a_c$, for any $\V>0$},
\end{cases}
\end{align*}
where
\begin{align*}
a_0=\frac{\S}{1-p},\quad a_c=\frac{\S+2(\nu+1)_+}{1-p}.
\end{align*}
It should be noted that, in the case $a<a_c$, the coefficient on the left estimate is independent of $l_1$ and that on the right is weakly depending on $l_2$. Thus, if the initial condition grows not too rapidly at infinity ($a<a_c$), then the grow-up of solutions to (\ref{Eqn:Main}) is, asymptotically, controlled totally by the non-autonomous unbounded sublinear source. On the other hand, at the critical growth exponent ($a=a_c$), the source and the initial condition are cooperated in controlling the asymptotic grow-up of solutions. Finally, if the initial condition grows sufficiently rapidly ($a>a_c$) the grow-up of solutions as $t\to\I$ is controlled totally by the initial condition. The effect of the third order viscous term $\Lap\P_tu$ is discussed in Remark \ref{Rem:VisEf}.  The result in this section is inspired by corresponding works on the sublinear heat equations, see for instance \cite{CazDickEscoWeiss01,FilaWinklerYanagida04,GuiNiWang01,WangYin12}.

\section{Preliminaries}\label{Sec:Prelim}

\subsection{Basic results, function spaces, and notation.} 

We will need the following basic facts.

\begin{lemma}\label{Technical1}
Let $0<p<1$.
\begin{enumerate}
\item[(1)] If $x,y\geq0$ then $|x^p-y^p|\leq|x-y|^p$.\vspace{5pt}
\item[(2)] If $g,\L\in C([0,T])$ with $g,\L\geq0$ satisfy
\begin{align*}
g(t)\leq\int_0^t\L(\t)g(\tau)^pd\tau\quad(t\in[0,T]),
\end{align*}
then 
\begin{align*}
g(t)\leq\left((1-p)\int_0^t\L(\t)d\t\right)^{\frac{1}{1-p}}\quad(t\in[0,T]).
\end{align*}
In particular, if $\L$ is a constant then $g(t)\leq((1-p)\L t)^{1/(1-p)}$. 
\end{enumerate}
\end{lemma}

\begin{proof}
Part (1) can be proved by elementary calculus. Part (2) follows directly from Bihari's inequality \cite{Bihari56} and that $F(x)= \int_0^xt^{-p}dt=x^{1-p}/(1-p)$ has the inverse $F^{-1}(x)=((1-p)x)^{1/(1-p)}$. 
\QED
\end{proof}

The phase spaces in this work are the spaces of {\it weighted continuous functions}:
\begin{align}
&BC_a=BC_a(\R^n)=\left\{\vp\in C(\R^n):|\vp(x)|\leq\mbox{const}\cdot|x|^{a}\,\,\mbox{as $|x|\to\I$}\right\},
\end{align}
where $a$ is any real number. $BC_a$ functions exhibit decay (or constant) at infinity if $a\leq0$, whereas they can have at most a power growth of order $a$ if $a>0$. It is easy to see that $BC_a\subsetneq BC_b$ for all $a<b$. The space $BC_a$ is endowed with a norm
\begin{align}
\|\vp\|_{R,a}=\left\|(R^2+|\cdot|^2)^{-\frac{a}{2}}\vp\right\|_{L^\I},
\end{align}
where $R>0$ is a constant. Any two norms $\|\cdot\|_{R,a}$, $\|\cdot\|_{\r,a}$ are equivalent, and $BC_a$ is a Banach space when equipped with any of these norms.

\begin{notation}\label{Not1}
Let $0<T\leq\I$, $a,d\in\R$, and $R,\r>0$. We denote
\begin{itemize}
\item[(i)] $\CZ_{T,a}=C\left([0,T);BC_a\right)$
\item[(ii)] $\displaystyle\|\vp\|_{\CZ_{T,a}}=\sup_{t\in[0,T)}\|\vp(\cdot,t)\|_{R,a}$
\item[(iii)] $Q_T=\R^n\times[0,T)$
\item[(iv)] $\LL_{\r,d}=(\r+|x|^2)^{\frac{d}{2}}$, $\LL=(R^2+|x|^2)^{-\frac{a}{2}}$
\item[(v)] $J_n=\{s\in\R:s(s+n-2)\geq0\}=\left(-\I,(2-n)_-\right]\cup\left[(2-n)_+,\I\right)$, $J_n^+=J_n\cap[0,\I)$, so
\begin{align*}
J_n=\begin{cases}
(-\I,0]\cup[1,\I)&\mbox{if $n=1$},\\
\R&\mbox{if $n=2$},\\
(-\I,2-n]\cup[0,\I)&\mbox{if $n\geq3$},
\end{cases}\quad J_n^+=\begin{cases}
\{0\}\cup[1,\I)&\mbox{if $n=1$},\\
[0,\I)&\mbox{if $n\geq2$}.
\end{cases}
\end{align*}
\item[(vi)] In our study of semigroup properties (see Section \ref{Sec:1.3}), we will use the following notation. If $w=w(x,t)$ and $\t\geq0$ we denote
\begin{align}\label{Not}
\begin{cases}
w^\tau(x,t)\triangleq w(x,t+\tau),\\
w^{-\t}(x,t)\triangleq w(x,[t-\t]_+),\\
w(\tau)\triangleq w(\cdot,\tau),\\
w^\ast\triangleq\CM w,\quad w^\dag\triangleq\CM^sw,
\end{cases}
\end{align}
where $\CM,\CM^s$ are defined by (\ref{Def:M}), (\ref{Def:Ms}) respectively. 
So $w(\t)$ is a function of $x$ and $w^\t,w^\ast,w^\dag$ are functions of $x,t$.
\item[(vii)] We also use the following standard notation: $f\lesssim g$ if $f\leq Cg$ for an (unimportant) constant $C>0$; $f\sim g$ if $f\lesssim g$ and $g\lesssim f$; $a_+=a\vee 0$ and $a_-=a\wedge0$ where $x\vee y=\max\{x,y\}$ and $x\wedge y=\min\{x,y\}$.
\end{itemize}


\end{notation}

\subsection{Bessel potential and (pseudoparabolic) Green operators}

Assuming all data are sufficiently regular, then Eq.\ (\ref{Eqn:Main}) is transformed via the Fourier transformation into the non-local equation
\begin{align*}
\P_tu=\CB\Lap u+\CB\left(Vu^p\right),\quad u|_{t=0}=u_0,
\end{align*}
where $\CB=(I-\Lap)^{-1}$ is the Bessel potential operator defined by 
\begin{align}
\CB u=\CF^{-1}\left(\frac{1}{1+|\xi|^2}\wh{u}\right)=B\ast u.
\end{align}
By the Duhamel's principle, then $u$ satisfies (\ref{Eqn:MildOrig}) where 
\begin{align}\label{Eqn:Green}
\begin{cases}
\displaystyle\CG(t)u=e^{-t}e^{t\CB}u=e^{-t}\sum_{k=0}^\I\frac{t^k}{k!}B_k\ast u=G(x,t)\ast u,\\
\displaystyle B_\nu(x)=\frac{2^{1-\nu}}{\Gm(\nu)}\left|x\right|^{\frac{2\nu-n}{2}}K_{\frac{n-2\nu}{2}}\left(\left|x\right|\right)\,\,\,\mbox{for all real $\nu>0$},
\end{cases}
\end{align}
and $K_\A$ is the modified Bessel function of the second kind. For more details, see \cite{AronSmith61,KaiNaumShish05,Karch99}.

\begin{lemma}\label{Lem:BasicBG}\

\begin{enumerate}
\item[(1)] $\CB$ and $\CG(t)$ are positive operators on $BC_a$.
\item[(2)] $\CB1=1$ and $\CG(t)1=1$ for all $t>0$.
\item[(3)] For each $d\in J_n$ (see Notation \ref{Not1} (v)), we have
\begin{align}\label{Est:BGd}
\CB\left(|x|^d\right)\geq|x|^d,\quad\CG(t)\left(|x|^d\right)\geq|x|^d.
\end{align}
\end{enumerate}

\end{lemma}

\begin{proof}
Part (1) is obvious and (2) can be found in \cite{Khomrutai15}.\smallskip


(3) As usual, let $r=|x|$. We have by direct calculation that
\begin{align*}
r^d-\Lap(r^d)&=r^d-(r^d)''-\frac{n-1}{r}(r^d)',\nonumber\\
&=r^{d-2}\left(r^2-d(d+n-2)\right)\leq r^d.
\end{align*}
Applying $\CB$ and using (1), we obtain $\CB(r^d)\geq r^d$. The second assertion is now immediate.\QED

\end{proof}

\begin{remark}
If $d\in\R\setminus J_n=((2-n)_-,(2-n)_+)$, then the estimates in (\ref{Est:BGd}) are reverse. 

\end{remark}

We prove a key result about the boundedness of $\CB$ and $\CG(t)$ on $BC_a$.

\begin{proposition}\label{Prop:BoundBG}
Let $\theta\in(0,1)$, $R>0$, $a\in\R$, and $\LL=(R^2+|x|^2)^{-\frac{a}{2}}$.
\begin{enumerate}
\item[(1)] If $a\leq0$ and $(1-\th)R^2\geq-a(n-(a+2)_-)$, then 
\begin{align}
\CB\LL\leq\theta^{-1}\LL.
\end{align}
\item[(2)] If $a\geq0$ and $(1-\th)R^2\geq a(n+(a-2)_+)$, then we have 
\begin{align}
\|\CB\vp\|_{R,a}&\leq\theta^{-1}\|\vp\|_{R,a},\quad
\|\CG(t)\vp\|_{R,a}\leq \B(t)\|\vp\|_{R,a},\quad\B(t)\triangleq e^{-(1-\theta^{-1})t},
\end{align}
for all $\vp\in BC_a$.

\end{enumerate}
\end{proposition}

\begin{proof}
(1) Let us denote $\LL=\VS^{-a/2}$ where $\VS=R^2+r^2$. Then we have
\begin{align*}
\LL-\Lap\LL&=\LL-\LL''-\frac{n-1}{r}\LL',\nonumber\\
&=\LL\left(1+an\VS^{-1}-a(a+2)r^2\VS^{-2}\right),\nonumber\\
&\geq\LL\left(1+anR^{-2}-\left(a(a+2)_-\right)R^{-2}\right)\geq\th\LL,
\end{align*}
where we have used that $a\leq0$, $0\leq\VS^{-1}\leq R^{-2}$, $0\leq r^2\VS^{-2}\leq\VS^{-1}\leq R^{-2}$, and the assumption on $R$. Taking $\CB$, we get $\LL=\CB(\LL-\Lap\LL)\geq\theta\CB\LL$ hence (1) is true.\smallskip

(2) Since $\|\vp\|_{R,a}=\|\LL\vp\|_{L^\I}$, we have
\begin{align*}
\CB\vp=\CB\left(\LL^{-1}\LL\vp\right)\leq\CB\left(\LL^{-1}\|\LL\vp\|_{L^\I}\right)=\CB(\LL^{-1})\|\vp\|_{R,a}.
\end{align*}
Since $a\geq0$, we get by (1) that $\CB(\LL^{-1})\leq\theta^{-1}\LL^{-1}$ provided $(1-\th)R^2\geq a(n-(-a+2)_+)=a(n+(a-2)_-)$. Thus
\begin{align*}
&\|\CB\vp\|_{R,a}=\left\|\LL\CB\vp\right\|_{L^\I}\leq\theta^{-1}\|\vp\|_{R,a},\\
&\|\CG(t)\vp\|_{R,a}\leq e^{-t}e^{t\|\CB\|}\|\vp\|_{R,a}\leq e^{-t+\frac{t}{\th}}\|\vp\|_{R,a},
\end{align*}
which are the desired estimates in (2).\QED
\end{proof}

\begin{remark}

\begin{enumerate}
\item[(i)] Using Proposition \ref{Prop:TwoSB} (1) and the same argument as above, it can be shown that 
\begin{align*}
\|\CG(t)\vp\|_{R,a}\leq(1+ct)^{\frac{a}{2}}\|\vp\|_{R,a},
\end{align*} 
for some constant $c=c(n,a)>0$, provided $R$ is large enough. This estimate is of course sharper than the one given in the proposition. For our work, however, the latter is enough.
\item[(ii)] If $a\in[2-n,0)$, of course $n\geq3$, then $\LL^{-1}-\Lap(\LL^{-1})\geq\LL^{-1}$, hence $\CB(\LL^{-1})\leq\LL^{-1}$. This implies
\[
\|\CB\vp\|_{R,a}\leq\|\vp\|_{R,a}.
\]
It is now clear that $\CG(t)(\LL^{-1})\leq\LL^{-1}$. Then $\|\CG(t)\vp\|_{R,a}
\leq\|\LL\CG(t)(\LL^{-1})\|_{L^\I}\|\vp\|_{R,a}$ and hence
\[
\|\CG(t)\vp\|_{R,a}\leq\|\vp\|_{R,a}.
\]
These two boundedness results are true for all $R>0$.
\item[(iii)] In our recent work \cite{KhomrutaiNA15} on superlinear pseudoparabolic equations (\ref{Eqn:Main}) with $p>1$, we have established the boundedness of $\CB$ and $\CG(t)$ on {\it weighted Lebesgue spaces}
\begin{align*}
L^{q,a}=\{\vp\in L^1_{loc}(\R^n):\|\vp\|_{L^{q,a}}\triangleq \|(1+|\cdot|^2)^{a/2}\vp\|_{L^q}<\I\}
\end{align*}
where $q\in[1,\I]$, $a\in\R$. The proof is rather involved. It was shown that
\begin{align*}
&\|\CB\vp\|_{L^{q,a}}\leq C_{n,a}\|\vp\|_{L^{q,a}},\\
&\|\CG(t)\vp\|_{L^{q,a}}\leq C_{n,q,a}\lng t\rng^{\frac{|a|}{2}}\|\vp\|_{L^{q,a}}\quad(t>0).
\end{align*}
This result, of course, generalizes Proposition \ref{Prop:BoundBG} (2). For the present investigation of sublinear equations, $BC_a$ is a more natural choice of function spaces the same as in the study of sublinear heat equation (\cite{AguirEscobedo86}). A novel idea in this work is that by using $\|\cdot\|_{R,a}$ instead of the usual norm $\|\cdot\|_{R=1,a}$, we gain a control over the operator norms of $\CB$ and $\CG(t)$. As was shown above, these linear operators behave almost as contractions when $\th\to1$ by choosing $R\gg1$.
\end{enumerate}

\end{remark}



We have the following pointwise lower estimate for the Bessel kernel.

\begin{lemma}\label{Lem:B1}

Let $n\geq1$. Then we have
\begin{align}
B(x)\geq b_0\phi_0(x)e^{-|x|}\quad\mbox{for all $x\in\R^n$},
\end{align}
where $b_0>0$ is a constant and
\begin{align}
\phi_0(x)=\begin{cases}
|x|^{\frac{1-n}{2}}&\mbox{if $n\neq2$ or $|x|\geq1$},\\
1-\ln|x|&\mbox{if $n=2$ and $|x|<1$}.
\end{cases}
\end{align}

\end{lemma}

\begin{proof}
This lemma has already appeared in \cite{Khomrutai15}, but for completeness we present a proof. It is well-known (see \cite{AronSmith61}, Lemma 4 in \cite{KhomrutaiNA15}) that $B$ is radial, strictly decreasing in $r=|x|$, and
\begin{align*}
B\sim\begin{cases}
r^{\frac{1-n}{2}}e^{-r}&\mbox{as $r\to\I$}\\
\xi(r)&\mbox{as $r\to0$}
\end{cases}\qquad\mbox{where}\quad\xi(r)=\begin{cases}
1&\mbox{if}\,\,n=1,\\
1-\ln r&\mbox{if}\,\,n=2,\\
r^{2-n}&\mbox{if}\,\,n\geq3.
\end{cases}
\end{align*}
Thus we can choose $C_1,C_2>0$ so that $B\geq C_1r^{\frac{1-n}{2}}e^{-r}$ for all $r\geq1$ and $B\geq C_2\xi(r)$ for all $0<r<1$.\smallskip

The desired estimate is trivial when $n=2$. If $n\neq2$, one can verify directly that $\xi\geq r^{\frac{1-n}{2}}e^{-r}$ for all $0<r<1$. Then by taking $C=\min\{C_1,C_2\}$, we obtain the desired estimate as well.\QED
\end{proof}

\subsection{Mild solutions, (mild) super and subsolutions; semigroup properties}\label{Sec:1.3}

We specify the notion of solutions considered in this work. 

\begin{definition}

Let $0<T\leq\I$ and $a\in\R$. By a positive \textbf{(mild) solution} of (\ref{Eqn:Main}) we mean $0\leq u\in \CZ_{T,a}$ which satisfies (\ref{Eqn:MildOrig}), or equivalently, 
\begin{align}
u=\CM u\quad\mbox{on $Q_T$},
\end{align}
where
\begin{align}\label{Def:M}
\CM u=\CM_{(u_0,V)}u\triangleq\CG(t)u_0+\int_0^t\CG(t-\t)\CB\left[Vu^p\right](x,\t)d\t.
\end{align}
In the case $T=\I$, $u$ is said to be a \textbf{global solution}.

\end{definition}

We will also need the following variant of the above definition. By formally setting 
\begin{align}
u=e^{-t}\mu
\end{align}
into (\ref{Eqn:Main}), we get the equation $\P_t\mu-\Lap\P_t\mu=\mu+e^{(1-p)t}V(x,t)\mu^p$
and the non-local equation
\begin{align}
\P_t\mu=\CB\mu+e^{(1-p)t}\CB\left[V\mu^p\right].
\end{align}

\begin{definition}

A function $u$ is said to be a positive \textbf{s-(mild) solution} of (\ref{Eqn:Main}) if $\mu\triangleq e^tu\in\CZ_{T,a}$ for some $0<T\leq\I$ and $a\in\R$, and $\mu$ satisfies (\ref{Eqn:MildMod}), or equivalently, 
\begin{align}
\mu=\CM^s\mu\quad\mbox{on $Q_T$},
\end{align}
where
\begin{align}\label{Def:Ms}
\CM^s\mu=\CM^s_{(u_0,V)}\mu\triangleq u_0+\int_0^t\CB \mu(\t)d\t+\int_0^te^{(1-p)\t}\CB\left[V\mu^p\right](x,\t)d\t.
\end{align}

\end{definition}

\begin{remark}
Under the assumptions in this work, nontrivial solutions to Eq.\ (\ref{Eqn:Main}) exhibit a {\it regularization effect} in the sense that, for any $t>0$, the solutions $u(t)\in BC_a$ with $a\geq0$, even if the initial condition $u_0\in BC_b$ with $b<0$, i.e.\ $u(t)$ is growing (in $x$) whereas $u_0$ is decaying! In view of that $BC_b\subset BC_a$, so the underlying function spaces, in the above definitions, can be taken to be $\CZ_{T,a}$ with $a\geq0$.
\end{remark}

\begin{remark}
\begin{enumerate}
\item[(i)] In the case that the data $V,u_0$ are sufficiently regular, e.g.\ in H\"older classes, and $V=V(x)$, it was shown in \cite{KhomrutaiNA15} every mild solution of the pseudoparabolic equation (\ref{Eqn:Main}) is classical. Thus the above two definitions coincide in this case.
\item[(ii)] It will be shown that the notion of s-mild solutions is `stronger' in the sense that every s-mild solution of (\ref{Eqn:Main}) is a mild solution. See Lemma \ref{Lem:TMisM} below.
\item[(ii)]  (Mild) \textit{supersolutions} (respectively, \textit{subsolutions}) are defined by requiring
\begin{align}
u\geq \CM u\quad\mbox{on $Q_T$}\quad (\mbox{resp.},\,\,u\leq\CM u\quad\mbox{on $Q_T$}).
\end{align}
We define s-(mild) super and subsolutions in a similar fashion using $\CM^s$.
\end{enumerate}

\end{remark}

\begin{lemma}\label{Lem:TMisM}

Let $0\leq u\in\CZ_{T,a}$ where $0<T\leq\I$ and $a\geq0$.
\begin{enumerate}
\item[(1)] If $u$ is an s-supersolution of (\ref{Eqn:Main}), then $u$ is a supersolution.
\item[(2)] If $u$ is an s-subsolution of (\ref{Eqn:Main}), then $u$ is a subsolution.
\end{enumerate}

\end{lemma}

\begin{proof}
A proof for part (1) can be found in \cite{Khomrutai15}. However, we sketch it again here for completeness sake and also because its idea will be used in proving part (2).\smallskip

(1) Recall $\CB$ is a positive operator. Setting $\mu=e^tu$ into the inequality $\mu\geq\CM^s\mu$, then $u$ satisfies
\begin{align}\label{MMM0}
u\geq e^{-t}u_0+\int_0^te^{-(t-\t)}\CB u(\t)d\t+\int_0^te^{-(t-\t)}\CB\left[Vu^p\right](\t)d\t
=v+\CF u+\CF w,
\end{align}
where 
\begin{align}
v=e^{-t}u_0,\quad \CF\vp\triangleq\int_0^te^{-(t-\t)}\CB\vp\,d\t,\quad\mbox{and}\quad w=Vu^p.
\end{align}
Since $\CB u\geq0$, we have $u\geq I_0\triangleq v+\CF w$. By iterating (\ref{MMM0}), $u\geq I_1\triangleq I_0+\CF I_0=(1+\CF)v+(\CF^2+\CF)w$, $u\geq I_0+\CF I_1=(1+\CF+\CF^2)v+(\CF^3+\CF^2+\CF)w\triangleq I_2$, and generally
\begin{align*}
u&\geq I_k\triangleq I_0+\CF I_{k-1}\nonumber\\
&=(1+\CF+\cdots+\CF^k)v+(\CF^{k+1}+\cdots+\CF)w\quad(k\geq1).
\end{align*}
We calculate $\CF^kv$ and $\CF^kw$. For any $k\geq1$, it is easy to see that
\begin{align*}
&\CF^kv=e^{-t}\frac{t^k}{k!}\CB^ku_0,\\
&\CF^{k}w=\int_0^te^{-(t-\t)}\frac{(t-\t)^{k-1}}{(k-1)!}\CB^{k}\left[Vu^p\right](\t)d\t,
\end{align*}
where the latter is true by Fubini's theorem. From the preceding three estimates, we conclude that 
\begin{align*}
u&\geq\lim_{k\to\I}I_k=\sum_{k=0}^\I \CF^kv+\sum_{k=1}^\I \CF^{k}w,\nonumber\\
&=\CG(t)u_0+\int_0^t\CG(t-\t)\CB\left[Vu^p\right](\t)d\t=\CM u,
\end{align*}
which is the desired result.\smallskip

(2) We perform an iteration similar to (1). Since $u$ is an s-mild sub-solution, it follows that $u$ satisfies
\begin{align*}
u\leq v+\CF u+\CF w=I_0+\CF u,
\end{align*}
where $v,\CF,w,I_0$ are the same as above. By iteration, $u\leq I_0+\CF(I_0+\CF u)=I_1+\CF^2u$, and generally
\begin{align*}
u\leq I_0+\CF(I_{k-1}+\CF^ku)=I_k+\CF^{k+1}u\quad(k\geq0),
\end{align*}
where $I_k$ is the same as in part (1). As above, $I_k\to\CM u$, so it suffices to show that 
\begin{align*}
\CF^ku\to0\quad\mbox{as $k\to\I$}.
\end{align*}

Let $t_0\in(0,T)$. We choose $\th\in(0,1)$ such that 
\begin{align*}
\theta^{-1}(1-e^{-t_0})=:\gamma<1,
\end{align*}
and $R>0$ sufficiently large according to Proposition \ref{Prop:BoundBG} (2). If $0<t\leq t_0$, then
\begin{align*}
\|\CF u(\t)\|_{R,a}&\leq\int_0^te^{-(t-\t)}\theta^{-1}\|u(\t)\|_{R,a}d\t,\nonumber\\
&\leq\gamma\sup_{[0,t_0]}\|u(\t)\|_{R,a}=\gamma\|u\|_{\CZ_{t_0,a}}.
\end{align*}
Thus we have $\|\CF v\|_{\CZ_{t_0,a}}\leq\gamma\|v\|_{\CZ_{t_0,a}}$ which implies
\begin{align*}
\|\CF^ku\|_{\CZ_{t_0,a}}\leq\gamma^k\|u\|_{\CZ_{t_0,a}}\to0\quad\mbox{as $k\to\I$}.
\end{align*}
Now $0\leq\LL \CF^ku\leq\|\CF^ku\|_{\CZ_{t_0,a}}\to0$ on $Q_{t_0}$. So $\CF^ku\to0$ pointwise as $k\to\I$.\QED
\end{proof}

\begin{corollary}\label{Cor:PimM}

If $u$ is an s-solution of (\ref{Eqn:Main}), then it is a solution. 

\end{corollary}

Finally, we present some semigroup properties for the problem (\ref{Eqn:Main}). 



\begin{lemma}\label{Lem:supersub}

Let $\t\in(0,T)$ and $0\leq u\in\CZ_{T,a}$.
\begin{enumerate}
\item[(1)] If $u$ is a supersolution (respectively, subsolution) of (\ref{Eqn:Main}) on $(0,T)$, then $u^{\t}$ satisfies
\begin{align}
u^{\t}\geq\CM_{\left(u^\ast(\t),V^{\t}\right)}u^{\t}\quad\mbox{on $Q_{T-\t}$}\quad\left(\mbox{resp.},\,\,u^\t\leq\CM_{\left(u^\ast(\t),V^\t\right)}u^\t\quad\mbox{on $Q_{T-\t}$}\right).
\end{align}
\item[(2)] If $u$ is a solution of (\ref{Eqn:Main}) on $(0,T)$, then $u^\t$ satisfies
\begin{align}
u^\t=\CM_{\left(u(\t),V^\t\right)}u^\t\quad\mbox{on $Q_{T-\t}$}.
\end{align}
\end{enumerate}

\end{lemma}

\begin{proof}
It suffices to prove (1). Let $u$ be a supersolution of (\ref{Eqn:Main}). So it satisfies $u(t+\t)\geq\CM u(t+\t)$ on $Q_{T-\t}$. Since
\begin{align*}
\CM u(t+\t)&=\CG(t+\t)u_0+\int_0^{t+\t}\CG(t+\t-s)\CB[Vu^p](s)ds,\nonumber\\
&=\CG(t)\left(\CG(\t)u_0+\int_0^\t\CG(\t-s)\CB[Vu^p](s)ds\right)+\int_\t^{t+\t}\CG(t+\t-s)\CB[Vu^p](s)ds,\nonumber\\
&=\CG(t)u^\ast(\t)+\int_0^{t}\CG(t-s)\CB[Vu^p](s+\t)ds\quad(\mbox{see (\ref{Not})}),\nonumber\\
&=\CM_{\left(u^\ast(\t),V^\t\right)}u^\t(t),
\end{align*}
it follows that $u^\t\geq\CM_{\left(u^\ast(\t),V^\t\right)}u^\t$ as claim. The second assertion is also true.\QED
\end{proof}

\begin{remark}\label{Rem:MsSem}
The semigroup property of s-sub (and s-super) solutions is a bit different. One can show that
\begin{align*}
\CM^s\mu(t+\t)=\CM^s_{\left(\mu^\dag(\t),e^{(1-p)\t}V^\t\right)}\mu^\t(t).
\end{align*}
If $u$ is an s-supersolution and $\T u\triangleq u^\t$, for a fixed $\t>0$, then $\T\mu\triangleq e^t\T u=e^{-\t}\mu^\t$ satisfies
\begin{align}
\T\mu\geq\CM^s_{\left(e^{-\t}\mu^\dag(\t),V^\t\right)}\T\mu\quad(\mbox{on $Q_{T-\t}$}),
\end{align}
where $\mu^\dag=\CM^s\mu$. Similar statements hold for s-solutions and s-subsolutions.
\end{remark}

\section{Global existence}\label{Sec:Glob}

We establish the global existence of solutions for the Cauchy problem (\ref{Eqn:Main}) under the assumption (\ref{HypVU}). Fix $\th\in(0,1)$ and $R\geq1$ such that Proposition \ref{Prop:BoundBG} is true and let $\t_0>0$. 
The assumption (\ref{HypVU}) implies that $V$ satisfies
\begin{align}\label{HypVUp}
0\leq V(x,t)\leq\Vt\left(\max_{[0,\t_0]}\L\right)\left(R^2+|x|^2\right)^{\frac{\S_+}{2}}\quad(\mbox{on $Q_{\t_0}$}),
\end{align}
for some constant $\Vt>0$ independent of $\t_0$. Let $\LL=\left(R^2+|x|^2\right)^{-a/2}$ where $a$ satisfies (\ref{HypAL}).


\begin{lemma}\label{Lem:LamV}

Assume (\ref{HypVU}) (see (\ref{HypVUp})) and (\ref{HypAL}). 
Then
\begin{align}
\LL^{-p}V\leq\Vt\left(\max_{[0,\t_0]}\L\right)\LL^{-1}\quad(\mbox{on $Q_{\t_0}$}).
\end{align}
Thus, for all $\vp\in BC_a$ and $t\leq\t_0$, we have
\begin{align}
\left\|\CG(t)\CB\left(V\vp^p\right)\right\|_{R,a}\leq\vU\B(t)\|\vp\|_{R,a}^p,\quad\vU=\vU(\t_0)\triangleq\theta^{-1}\Vt\left(\max_{[0,\t_0]}\L\right).
\end{align}

\end{lemma}

\begin{proof}
By (\ref{HypAL}), we have $\S_++ap\leq a$. Hence
\begin{align*}
\LL^{-p}V\leq \Vt\left(\max_{[0,\t_0]}\L\right)\left(R^2+|x|^2\right)^{\frac{\S_++ap}{2}}\leq\Vt\left(\max_{[0,\t_0]}\L\right)\LL^{-1},
\end{align*}
where we have used (\ref{HypVUp}) and that $R\geq1$. By Proposition (\ref{Prop:BoundBG}) (2) and the preceding estimate, we have
\begin{align*}
\left\|\CG(t)\CB\left(V\vp^p\right)\right\|_{R,a}&\leq\theta^{-1}\B(t)\left\|\LL V\vp^p\right\|_{L^\I},\\
&\leq\theta^{-1}\Vt\left(\max_{[0,\t_0]}\L\right)\B(t)\|(\LL\vp)^p\|_{L^\I}\leq\vU\B(t)\|\vp\|_{R,a},
\end{align*}
which is the desired estimate. \QED
 
\end{proof}

Now we prepare to solve the problem. Observe that $0\leq u\in\CZ_{\t_0,a}$ is a solution of (\ref{Eqn:Main}) if and only if the function
\begin{align}
w=\LL u\in\CZ_{\t_0,0}
\end{align}
satisfies the integral equation
\begin{align}\label{Eqn:wp}
w(t)=\LL\CG(t)u_0+\LL\int_0^t\CG(t-\t)\CB\left(\LL^{-p}Vw^p\right)(\t)\,d\t\quad(\mbox{on $Q_{\t_0}$}).
\end{align}
The function $w\mapsto w^p$ is not Lipschitz continuous near $w=0$ when $0<p<1$, so we cannot employ the standard contraction mapping technique directly to solve (\ref{Eqn:wp}). An approximation argument is needed (see for instance \cite{AguirEscobedo86},\cite{WuZhaoYinLi01}).\smallskip

We consider $F:\R\to\R$ such that
\begin{align}\label{Assump:Phi}
F\in\mbox{Lip}(\R),\qquad F'\geq0,\qquad F(0)=0,
\end{align}
and then solve the following problem: Find $0\leq w\in\CZ_{\t_0,0}$ satisfying
\begin{align}\label{Eqn:Komg}
\begin{cases}
\displaystyle w=\CK w\quad(\mbox{on $Q_{\t_0})$, where}\\
\vspace{-10pt}\\
\displaystyle\CK w(t)\triangleq\LL\CG(t)u_0+\LL\int_0^t\CG(t-\t)\CB\left(\LL^{-p}VF(w)\right)(\t)\,d\t.
\end{cases}
\end{align}
Whenever the dependence on $u_0,F$ is important, $\CK$ will be denoted by $\CK_{(u_0,F)}$.

\begin{proposition}\label{Prop:wcomp}

Assume (\ref{HypVU}) (see (\ref{HypVUp})), (\ref{HypAL}), and (\ref{Assump:Phi}). If $u_0\in BC_a$ then Eq.\ (\ref{Eqn:Komg}) admits a unique solution $w\in\CZ_{\t_0,0}$, and if $u_0\geq0$ then $w\geq0$. In particular,  the solution can be extended globally in time to $w\in\CZ_{\I,a}$.

\end{proposition}

\begin{proof}

Fix $\theta\in(0,1)$ and $R\geq1$ satisfying Proposition (\ref{Prop:BoundBG}). 
For convenience, we use $\|\cdot\|=\|\cdot\|_{R,a}=\|\LL\cdot\|_{L^\I}$. 
By (\ref{Assump:Phi}), there is a constant $l>0$ such that 
\begin{align*}
|F(s)|\leq l|s|\quad\mbox{for all $s\in\R$}.
\end{align*}
First we solve the local problem. Let $T>0$. The Banach space $\CX\triangleq\CZ_{T,0}$ is equipped with the norm
\begin{align*}
\|w\|_{\CX}=\sup_{[0,T]}\|w(t)\|_{L^\I}.
\end{align*}
$w$ is a solution of (\ref{Eqn:Komg}) on $[0,T]$ if and only it is a fixed point of $\CK$ on $\CX$.\smallskip

Let us show that $\CK$ is a self-map on $\CX$. For $w\in\CX$, we have by Proposition \ref{Prop:BoundBG} (2) and Lemma \ref{Lem:LamV} that
\begin{align*}
\|\CK w(t)\|_{L^\I}&\leq\|\CG(t)u_0\|+\int_0^t\left\|\CG(t-\t)\CB\left(\LL^{-p}VF(w)\right)(\t)\right\|d\t,\nonumber\\
&\leq\B(t)\|u_0\|+\vU\int_0^t\B(t-\t)\|F(w)(\t)\|_{L^\I}d\t,
\nonumber\\
&\leq\B(t)\|u_0\|+\vU l\int_0^t\B(t-\t)\|w(\t)\|_{L^\I}d\t,\nonumber\\
&\leq\B(T)\left(\|u_0\|+\vU lT\|w\|_\CX\right).
\end{align*}
Thus $\CK w(t)\in L^\I$ for all $t\in[0,T]$. It is also clear that $\CK w(t)\in C(\R^n)$ so $\CK w(t)\in BC_0$.
For $s,t\in(0,T]$, one can argue as above to get that
\begin{align*}
\|\CK w(s)-\CK w(t)\|_{L^\I}&\leq\|(\CG(t)-\CG(s))u_0\|+\vU l\B(T)\|w\|_\CX|s-t|\to0
\end{align*}
as $|s-t|\to0$. So $t\mapsto\CK w(t)$ is continuous, hence $\CK w\in\CX$. Also, note that $\CK$ satisfies the estimate
\begin{align}\label{Eqn:QWX}
\|\CK w\|_\CX\leq\B(T)\left(\|u_0\|+\vU lT\|w\|_\CX\right).
\end{align}

Let $D>0$ be a number to be specified and define 
\begin{align}
\CX_D=\{w\in\CX:\|w\|_\CX\leq D\}.
\end{align}
We show that $\CK$ is a self-map on $\CX_D$ provided $D$ is sufficiently large and $T$ is sufficiently small. Let $w\in\CX_D$. By (\ref{Eqn:QWX}), we have $\|\CK w\|_\CX\leq D$, provided $T,D$ satisfy
\begin{align*}
\B(T)\|u_0\|\leq\frac{D}{2}\quad\mbox{and}\quad\vU l\B(T)T\leq\frac{1}{2}.
\end{align*}
The second condition is true if $T>0$ is chosen sufficiently small and it can be achieved depending only upon $\vU,\theta,l$ and is independent of $u_0$. Next we choose $D>0$ large, depending upon $T$ and $u_0$, to fulfil the first condition. Fix such a choice of $D,T$ then we obtain $\CK:\CX_D\to\CX_D$.\smallskip

We show that $\CK$ is a contraction. Let $w,\omg\in\CX_D$. By the property of $F$ and Lemma \ref{Lem:LamV}, then
\begin{align*}
\|\CK w-\CK\omg\|_\CX&\leq\sup_{t\in[0,T]}\int_0^t\left\|\CG(t-\t)\CB\left(\LL^{-p}V|F( w)-F(\omg)|\right)(\t)\right\|d\t,\nonumber\\
&\leq\sup_{t\in[0,T]}\vU \B(t)\int_0^t\left\|\left(F( w)-F(\omg)\right)(\t)\right\|_{L^\I}d\t,\nonumber\\
&\leq \vU l\B(T)T\| w-\omg\|_\CX\leq\frac{1}{2}\| w-\omg\|_\CX.
\end{align*}
Hence, $\CK:\CX_D\to\CX_D$ is a contraction as desired.\smallskip

By the Banach fixed point theorem, there is a unique mild solution $ w$ to the Cauchy problem (\ref{Eqn:Komg}) which is defined on the time interval $[0,T]$. Note that $T$ depends only on $\Vt,\theta,l$ and is independent of $u_0$.\smallskip

The semigroup property of the integral equation $w=\CK w$ is
\begin{align*}
w^T(t)=\LL\CG(t)\left(\LL^{-1}w(T)\right)+\LL\int_0^t\CG(t-\t)\CB\left(\LL^{-p}V^TF(w^T)\right)(\t)\,d\t\quad(\mbox{on $Q_{\t_0-T}$}),
\end{align*}
which is almost the same as the equation $w=\CK w$ except for the replacement $u_0\to\LL^{-1}w(T)$ and $V\to V^T$, where $V^T$, $w^T$ are the translations by time $T$. Since the local existence time $T$ for the equation $w=\CK w$ is independent of $u_0$ and the potential $V^T$ satisfies (\ref{HypVUp}) on $Q_{\t_0-T}$, we can apply the preceding result to obtain a unique solution $w$ on
\[
[T,2T],\,\,[2T,3T],\,\,\ldots\quad\mbox{and so forth}.
\]
Hence the mild solution $w$ can be extended to $[0,\t_0]$.\smallskip

The assertion that $w\geq0$ provided $u_0\geq0$ follows from Picard iteration and (\ref{Assump:Phi}). See also the following proposition. Applying the above semigroup property and the existence result on $[0,\t_0]$ for any $\t_0>0$, the solution $w$ can be extended to $[\t_0,2\t_0],[2\t_0,3\t_0]$, and so on. Therefore $w$ exists globally in time.\QED

\end{proof}

Next we prove a comparison result for Eq.\ (\ref{Eqn:wp}).

\begin{proposition}\label{Prop:CompW}

Assume (\ref{HypVU}) (see (\ref{HypVUp}), (\ref{HypAL}), and (\ref{Assump:Phi}). If $w\geq0$ and $\omg\geq0$ satisfy
\begin{align}
w&\geq\CK_{(u_0,V)}w\quad\mbox{on $Q_{\t_0}$},\quad\omg\leq\CK_{(v_0,V)}\omg\quad\mbox{on $Q_{\t_0}$},
\end{align}
and $u_0\geq v_0$ on $\R^n$, then $w\geq\omg$ on $Q_{\t_0}$. In particular, if $u_0\geq0$ then $w\geq0$.

\end{proposition}

\begin{proof}
Since $v_0\leq u_0$, we have $\LL\CG(t)(v_0-u_0)\leq0$. Since $(F(\omg)-F(w))_+\leq l(\omg-w)_+$, we have
\begin{align*}
(\omg-w)(t)&\leq\LL\int_0^t\CG(t-\t)\CB\left(\LL^{-p}V\left(F(\omg)-F(w)\right)_+\right)(\t)d\t,\nonumber\\
&\leq l\int_0^t\left\|\CG(t-\t)\CB\left(\LL^{-p}V(\psi-w)_+\right)(\t)\right\|_{R,a}d\t,\nonumber\\
&\leq \vU l\int_0^t\B(t-\t)\|(\omg-w)_+(\t)\|_{L^\I}d\t\quad(\mbox{by Lemma \ref{Lem:LamV}}),\nonumber\\
&\leq \vU l\B(t)\int_0^t\|(\omg-w)_+(\t)\|_{L^\I}d\t.
\end{align*}
Let $\zeta(t)=\|(\omg-w)_+(t)\|_{L^\I}$. Then the continuous function $\zeta$ satisfies the differential inequality 
\begin{align*}
\zeta(t)\leq\vU l\B(t)\int_0^t\zeta(s)\,ds.
\end{align*}
By Gronwall's inequality, we conclude that $\zeta\equiv0$. Therefore $w\geq\omg$.\QED
\end{proof}

The following elementary result will also be useful.

\begin{lemma}\label{Lem:Gt1}

Let $a\in\R$. If $\vp\in BC_a$ and $\vp\geq0$, then
\begin{align*}
\CG(t)\vp\geq e^{-t}\vp\quad\mbox{on $Q_\I$}.
\end{align*}

\end{lemma}

\begin{proof}
Since $u=\CG(t)\vp\geq0$ solves the linear problem $\P_tu-\Lap\P_tu=\Lap u$, $u(\cdot,0)=\varphi$, in the sense of mild solutions. So the function $\mu=e^tu$ satisfies
\begin{align*}
\mu(t)=\vp+\int_0^t\CB\mu(\t)\,d\t.
\end{align*}
Since $\mu\geq0$, we have $\mu\geq\vp$. Hence $\CG(t)\vp=e^{-t}\mu\geq e^{-t}\vp$ as needed.\QED
\end{proof}

We can now state and prove our main global existence result.

\begin{theorem}[Global existence]\label{Thm:Existence}

Assume (\ref{HypVU}) (see (\ref{HypVUp})) and (\ref{HypAL}). If $u_0\in BC_a$ with $u_0\geq0$, then (\ref{Eqn:Main}) has a global solution 
\[
0\leq u\in\CZ_{\I,a}.
\]

\end{theorem}

\begin{remark}\label{Rem:Existence}
In this theorem, if $u_0\in BC_a$ with $a<a_0\triangleq\frac{\S_+}{1-p}$, i.e.\ (\ref{HypAL}) is not true, we still can apply the result of this theorem to get a global mild solution 
\[
0\leq u\in\CZ_{\I,a_0},
\]
because $u_0\in BC_{a_0}$. The path $t\mapsto u(t)$ is continuous in $BC_{a_0}$ but it is not continuous when considered in $BC_a$. In fact, as will be seen in the next section, the solution lies in $BC_{a_0}\setminus BC_a$ for all $t>0$. In particular, if $\S>0$ and $a\leq0$, the unbounded sublinear source induce an instantaneous growth even when the initial condition is a decay (or bounded) function.

\end{remark}

\begin{proof}[\textbf{Theorem \ref{Thm:Existence}}]
For each $m\in\BN$, let $F_m$ be a function satisfying (\ref{Assump:Phi}) and $F_m(s)=s^p$ for all $s\geq m^{-2}$. One can choose for example,
\[
F_m(s)=\begin{cases}\displaystyle
m^{2(1-p)}s&\mbox{if $0\leq s\leq m^{-2}$},\\
\vspace{-10pt}\\
s^p&\mbox{if $s\geq m^{-2}$},\\
\vspace{-10pt}\\
-F_m(-s)&\mbox{if $s<0$}.
\end{cases}
\]

Fix $u_0\in BC_a$ and let $u_{0,m}=u_0+\frac{1}{m\LL}\in BC_a$. Let $0\leq w_m\in\CZ_{\I,0}$ be the unique mild solution of (\ref{Eqn:Komg}) with $F,u_0$ replaced by $F_m,u_{0,m}$. This means $w_m$ satisfies 
\begin{align}
w_m(t)=\LL\CG(t)u_{0,m}+\LL\int_0^t\CG(t-\t)\CB\left(\LL^{-p}VF_m(w_m)\right)(\t)d\t.
\end{align}
Since $F_m(w_m)\geq0$ and $u_{0,m}\geq\frac{1}{m\LL}$, we have by Lemma \ref{Lem:Gt1} that
\begin{align}
w_m\geq\LL\CG(t)u_{0,m}\geq\frac{e^{-t}}{m}\quad(m\geq1).
\end{align}

\begin{claim}
For each $\t_0>0$, there is $N>0$ such that $\{w_m\}_{N}^\I$ is pointwise non-increasing on $Q_{\t_0}$.
\end{claim}

\begin{proof}
Choose $N\geq e^{\t_0}$. Let $(x,t)\in Q_{\t_0}$. If $m\geq N$ then $m\geq e^t$ so
\begin{align}\label{IdenWm1}
w_m\geq\frac{1}{m^2}\quad(\mbox{on $Q_{\t_0}$}).
\end{align}
By the choice of $F_m$, we have $F_m(s)=s^p$ for all $s\geq\frac{1}{m^2}$. This implies
\begin{align}\label{IdenWm2}
\begin{cases}
F_m(w_m)=w_m^p,\,\,\mbox{and}\\
F_\ell(w_m)=w_m^p=F_m(w_m)
\end{cases}
\end{align}
on $Q_{\t_0}$ for all $\ell>m\geq N$. The latter means that $w_m,w_\ell$ satisfy $w_m=\CK_{(u_{0,m},F_\ell)}w_m$ and $w_\ell=\CK_{(u_{0,\ell},F_\ell)}$ on $Q_{\t_0}$. Since $u_{0,\ell}\leq u_{0,m}$, it follows by Proposition (\ref{Prop:CompW}) that $w_\ell\leq w_m$ on $Q_{\t_0}$. \QED
\end{proof}

In view of the claim, we can define the function $w:Q_\I\to\R$ by the pointwise limit:
\begin{align}
w(x,t)=\lim_{m\to\I}w_m(x,t).
\end{align}
Clearly 
\[
w\in L^\I_{loc}\left([0,\I);L^\I(\R^n)\right)\quad\mbox{and}\quad w\geq0.
\]

\begin{claim}
The function $w$ satisfies (\ref{Eqn:wp}) on $Q_{\I}$.
\end{claim}

\begin{proof}
Let $\t_0>0$ and $m\geq N\geq e^{\t_0}$. Then we have, on $Q_T$, by (\ref{IdenWm2}) that
\begin{align}\label{IdenWm3}
w_m(t)&=\LL\CG(t)u_0+\frac{1}{m}\LL\CG(t)\LL^{-1}+\LL\int_0^t\CG(t-\t)\CB\left(\LL^{-p}Vw_m^p\right)(\t)d\t\quad(\mbox{on $Q_{\t_0}$}).
\end{align}
By the proof the first claim, we have $w_N\geq 1/N^{2}$ hence $w_N^p=w_N^{p-1}w_N\leq N^{2(1-p)}w_N$. This implies that the integral on the right hand side of (\ref{IdenWm3}) when $m=N$ is convergent:
\begin{align}
\LL\int_0^t\CG(t-\t)\CB\left[\LL^{-p}Vw_N^p\right](\t)\,d\t\leq\vU N^{2(1-p)}\t_0\B(\t_0)\|w_{N}\|_{\CZ_{\t_0,0}}<\I.
\end{align}

We take $m\to\I$ in (\ref{IdenWm3}) pointwise. The left hand side converges to $w$. On the right hand side, the first term is constant while the second term converges to zero. For the third term, we apply the monotone convergence theorem to conclude that it converges to $\LL\int_0^t\CG(t-\t)\CB\left(\LL^{-p}Vw^p\right)(\t)\,d\t$. Thus, as $m\to\I$,
\begin{align}
w=\LL\CG(t)u_0+\LL\int_0^t\CG(t-\t)\CB\left(\LL^{-p}Vw^p\right)(\t)\,d\t\quad(\mbox{on $Q_{\t_0}$}).
\end{align}
This is true for arbitrary $\t_0>0$ so the claim is true.\QED
\end{proof}

We finish the proof of Theorem \ref{Thm:Existence}. We set $u=\LL^{-1}w\geq0$. It was shown that $\LL u(t)\in L^\I$ and $u$ satisfies $u=\CM u$ on $Q_\I$. Using Proposition \ref{Prop:BoundBG} (2), one can readily check that $\CM u(t)\in BC_a$. Finally, one can argue as in Proposition \ref{Prop:wcomp} to find that $t\mapsto\CM u(t)$ is continuous. Therefore $u=\CM u\in C([0,\I);BC_a)$ is a mild solution of the Cauchy problem (\ref{Eqn:Main}).\QED
\end{proof}

\section{Lower grow-up rate}\label{Sec:Grow}

In this section, we prove a lower grow-up result for solutions of (\ref{Eqn:Main}). We assume that $u_0$ is a non-negative continuous function and $u_0\neq0$. 
We also assume the potential $V$ satisfies (\ref{HypVL}), i.e.\
\begin{align}\label{HypVLp}
V(x,t)\geq\L(t)|x|^\S\quad(|x|\geq1,0<t<\t_0),
\end{align}
where $\L(t)>0$ is a continuous function and $\S\in J_n^+$. Here it is no loss of generality in considering ``$|x|\geq1$" instead of a slightly more general case that ``$|x|\geq R_0$" for some $R_0>0$. See also Remark \ref{Rem:AftLow} below.

\begin{lemma}\label{Lem:lowert0}

If $0\leq\mu\in\CZ_{T,a}$ satisfies 
\begin{align}\label{GrowLem1}
\mu(t)\geq u_0(x)+\int_0^t\CB\mu(\t)d\t\quad(\mbox{on $Q_T$}),
\end{align}
then, for $\D>1$ and $0<t_0<T$, there is a constant $C_0=C_0(\D,t_0,u_0)>0$ such that 
\begin{align}
\int_0^{t_0}\CB\mu(\t)d\t\geq C_0e^{-\D|x|}\quad(\mbox{on $\R^n$}).
\end{align}
In particular, if $u$ is an s-supersolution of (\ref{Eqn:Main}) and $\mu=e^tu$, then
\begin{align}
\mu(t_0)\geq\CM^s\mu(t_0)\geq C_0e^{-\D|x|}\quad(\mbox{on $\R^n$}).
\end{align}

\end{lemma}

\begin{proof}
Since $\CB$ is a positive operator and $\mu\geq0$, we have by (\ref{GrowLem1}) that $\mu\geq u_0$ on $Q_T$. Since $u_0\geq0$ and $u_0\neq0$, we have $\CB u_0=B\ast u_0>0$ on $Q_T$. We define
\begin{align*}
\A_0\triangleq\frac{t_0}{2}\min_{|x|\leq2}\CB u_0\quad\Rightarrow\quad\A_0\leq\min_{|x|\leq2,\,t_0/2\leq t\leq t_0}\int_0^t\CB\mu(x,\t)d\t.
\end{align*}
Clearly $\A_0>0$. If $|x|\leq2$ then $\int_0^{t_0}\CB\mu(x,\t)d\t\geq \A_0\geq C_1e^{-\D|x|}$. By (\ref{GrowLem1}), it follows that $\mu(y,t)\geq \A_0$ for all $|y|\leq1$, $t_0/2\leq t\leq t_0$. If $|x|>2$ then we have by Lemma \ref{Lem:B1} that
\begin{align*}
\int_0^{t_0}\CB\mu(x,\t)ds&\geq\int_{t_0/2}^{t_0}\int_{|y|<1}B(x-y)\mu(y,\t)\,dyd\t,\nonumber\\
&\geq \A_0b_0\frac{t_0}{2}\int_{|y|<1}|x-y|^{-\frac{n-1}{2}}e^{-|x-y|}dy\qquad(\because\,|x-y|>1),\nonumber\\
&\geq \A_0b_0\frac{t_0}{2}e^{-(|x|+1)}\int_{|y|<1}|x-y|^{-\frac{n-1}{2}}dy,\nonumber\\
&\geq \A_0b_0\frac{t_0}{2}e^{-|x|}(1+|x|)^{-\frac{n-1}{2}}\geq C_2e^{-\D|x|}.
\end{align*}
Taking $C_0=C_1\wedge C_2$, the desired estimate then follows.\QED
\end{proof}




\begin{lemma}\label{Lem:BesselLower}

Assume $V$ satisfies (\ref{HypVL}). There are a constant $\V_0>0$ and a non-increasing function $\eta:\R\to\R$, $\eta>0$, 
such that the following properties hold. For any $P\geq0$, we have
\begin{align}\label{Est:BeP}
\CB\left(Ve^{-P|\cdot|}\right)\geq \eta(P)\L(t)e^{-P|x|}\quad(\mbox{on $Q_{\t_0}$}),
\end{align}
and for any $d\geq0$ we have
\begin{align}\label{Est:BVd}
\CB\left(V|x|^d\right)\geq\V_0\L(t)|x|^{\S+d}\quad(\mbox{on $Q_{\t_0}$}).
\end{align}
If $d$ is bounded, then $\V_0>0$ can be chosen independent of $d$.

\end{lemma}

\begin{proof}
By Lemma \ref{Lem:B1}, the Bessel potential kernel satisfies $B\geq b_0\phi_0(x)e^{-|x|}>0$. Using (\ref{HypVLp}), then we have
\begin{align*}
&\CB\left(Ve^{-P|\cdot|}\right)=\int B(y)V(x-y,t)e^{-P|x-y|}dy,\nonumber\\
&\hphantom{\CB\left(Ve^{-P|\cdot|}\right)}\geq b_0\L(t)\int_{|x-y|\geq1}\phi_0(y)e^{-|y|}|x-y|^\S e^{-P(|x|+|y|)}dy,\nonumber\\
&\hphantom{\CB\left(Ve^{-P|\cdot|}\right)}\geq b_0\L(t)e^{-P|x|}\int_{|x-y|\geq1}\phi_0(y)e^{-(1+P)|y|}|x-y|^\S dy=K\L(t)e^{-P|x|},\\
&\mbox{where}\,\, K\triangleq b_0\int_{|x-y|\geq1}\phi_0(y)e^{-(1+P)|y|}|x-y|^{\S}dy.
\end{align*}
If $|x|\leq2$ and $|y|\geq4$, then $|x-y|\geq|y|-|x|\geq\frac{|y|}{2}>1$, hence 
\begin{align*}
K&\geq b_0\omg_n\int_4^\I\phi_0(r)e^{-(1+P)r}\left(\frac{r}{2}\right)^\S r^{n-1}dr=:\eta_1(P).
\end{align*}
If $|x|\geq2$ and $|y|\leq\frac{|x|}{2}$, then $|x-y|\geq|x|-|y|\geq\max\{1,|y|\}$, hence
\begin{align*}
K&\geq b_0\omg_n\int_0^{1}\phi_0(r)e^{-(1+P)r}r^{\S+n-1}dr
=:\eta_2(P).
\end{align*}
$\eta_1,\eta_2$ are finite because $\int_0^\I\phi_0(r)e^{-(1+P)r}r^{n+\S-1}dr<\I$. Also, $\eta_1,\eta_2>0$. Setting $\eta=\min\{\eta_1,\eta_2\}$, the estimate (\ref{Est:BeP}) then follows.\smallskip

To prove (\ref{Est:BVd}), we employ as above to get
\begin{align*}
\CB\left(V|x|^d\right)\geq\L(t)K,\quad K\triangleq b_0\int_{|x-y|\geq1}\phi_0(y)e^{-|y|}|x-y|^{\S+d}dy.
\end{align*}
If $|x|\leq2$ and $|y|\geq4$ then $|x-y|>1\geq\frac{|x|}{2}$ hence 
\[
K\geq \V_1|x|^{\S+d},\quad\V_1\triangleq b_0\omg_n2^{-(\S+d)}\int_4^\I\phi_0(r)e^{-r}r^{n-1}dr.
\]
If $|x|\geq2$ and $|y|\leq\frac{|x|}{2}$, then $|x-y|\geq\frac{|x|}{2}\geq1$, hence
\[
K\geq\V_2|x|^{\S+d},\quad\V_2\triangleq b_0\omg_n2^{-(\S+d)}\int_0^1\phi_0(r)e^{-r}r^{n-1}dr.
\]
Taking $\V_0=\min\{\V_1,\V_2\}$ the estimate (\ref{Est:BVd}) then follows.\QED
\end{proof}

\begin{remark} \label{Rem:AftLow}
\begin{enumerate}
\item[(i)] Observe that, in the assumption (\ref{HypVLp}), we assume no control of $V(\cdot,t)$ on $B_1(0)$. If (\ref{HypVLp}) holds on $\R^n\times(0,\t_0)$ instead, then we can apply the fact that $\CB(|x|^d)\geq|x|^d$ when $d\in J_n^+$ (Lemma (\ref{Lem:BasicBG}) to deduce that $\V_0=1$. Also note that $\V_0=\V_0(R_0)$ if the  assumption (\ref{HypVLp}) is assumed for $|x|\geq R_0$. 
\item[(ii)] The results of the above lemma and the following theorem can be generalized to a potential $V$ satisfying 
\begin{align}
V(x,t)\geq\sum_{k=1}^N\L_k(t)|x-x_k|^{\S_k}\quad(\forall\,k,|x-x_k|\geq R_k,0<t<\t_0),
\end{align}
where $x_k\in\R^n,\S_k\in J_n^+$, and $\L_k(t)>0$ are continuous functions, for all $k$.
\end{enumerate} 

\end{remark}

We prove the following preliminary version of lower {\it grow-up rate} of solutions of (\ref{Eqn:Main}).

\begin{theorem}[Lower grow-up estimate]\label{Thm:uniflower}

Assume (\ref{HypVL}) (see (\ref{HypVLp}) and $0\leq u_0\in C(\R^n)$, $u_0\neq0$. If $u\geq0$ is an s-supersolution of Eq.\ (\ref{Eqn:Main}) on $Q_{\t_0}$, then $u$ satisfies
\begin{align}
u(x,t)\geq\left(\V_0(1-p)|x|^\S\int_0^t\L(\t)d\t\right)^{\frac{1}{1-p}}\quad(\mbox{on $Q_{\t_0}$}),
\end{align}
where $\V_0>0$ is the constant occurred from the estimate (\ref{Est:BVd}).

\end{theorem}

Let us use the notation
\begin{align}\label{Def:Last}
\L_\ast(t)=\int_0^t\L(\t)\,d\t.
\end{align}

\begin{proof}
Let $\D>1$ be such that $P\triangleq\D p\in(0,1)$. We split the proof into several steps. In Step 1-4, we establish the theorem assuming that 
\begin{align*}
u_0(x)\geq C_0e^{-\D|x|}\quad(\mbox{on $\R^n$}),
\end{align*}
for some constant $C_0>0$. The general case will be proved in Step 5 using lemma (\ref{Lem:lowert0}). \smallskip

{\it Step 1.} Let $\mu=e^tu$ and $q=\frac{1}{1-p}$.

\begin{claim}\label{Claim:Lower}
We have 
\begin{align}
\mu(x,t)\geq\left(\V_0|x|^\S(1-p)\L_\ast(t)\right)^q\quad(\mbox{on $Q_{\t_0}$}).
\end{align}

\end{claim}

\begin{proof}[{\bf Claim \ref{Claim:Lower}}]
Fix $(x,t)\in Q_{\t_0}$. The claim will be proved by repeatedly applying the fact that
\begin{align}\label{NonLins}
\mu\geq\CN^s\mu\triangleq\int_0^t\CB\left(V\mu^p\right)(\t)\,d\t,
\end{align}
which is true because $\mu\geq\CM^s\mu$ and $e^{(1-p)\t}\geq1$. Since $u$ is an s-supersolution, we have $\mu\geq u_0\geq C_0e^{-\D|x|}$. In $Q_T$, using Lemma \ref{Lem:BesselLower} we have
\begin{align*}
&\mu\geq\CN^s\left(C_0e^{-\D|\cdot|}\right)=C_0^p\int_0^t\CB\left[Ve^{-P|\cdot|}\right](\t)\,d\t,\nonumber\\
&\hphantom{\mu}\geq C_0^p\eta(P)e^{-P|x|}\int_0^t\L(\t)d\t,\nonumber\\
&\hphantom{\mu}=C_1\L_1(t)e^{-P|x|}\quad\mbox{where}\,\,C_1\triangleq C_0^p\eta(P),\quad\L_1(t)\triangleq\L_\ast(t).
\end{align*}
Using this estimate, the fact that $0<p<P<1$, and Lemma \ref{Lem:BesselLower}, we perform an iteration to get
\begin{align*}
&\mu\geq \CN^s\left(C_1\L_1(t)e^{-P|\cdot|}\right)\geq C_1^p\int_0^t\L_1(\t)^p\CB\left[Ve^{-P^2|\cdot|}\right](\t)\,d\t,\nonumber\\
&\hphantom{\mu}\geq C_1^p\eta\left(P^2\right)e^{-P^2|x|}\int_0^t\L(\t)\L_1(\t)^pd\t,\nonumber\\
&\hphantom{\mu}=C_2\L_2(t)e^{-P^2|x|}\quad\mbox{where}\,\,C_2\triangleq C_0^{p^2}\left[\eta(P)^p\eta\left(P^2\right)\right],\quad\L_2(t)=\int_0^t\L(\t)\L_1(\t)^pd\t.
\end{align*}
Continuing the iteration, we get
\begin{align*}
&\mu\geq\CN^s\left(C_2\L_2(t)e^{-P^2|\cdot|}\right)\geq C_2^p\int_0^t\L_2(\t)^p\CB\left[Ve^{-P^3|\cdot|}\right](\t)\,d\t,\nonumber\\
&\hphantom{\mu}\geq C_2^p\eta(P^3)e^{-P^3|x|}\int_0^t\L(\t)\L_2(\t)^pd\t,\nonumber\\
&\hphantom{\mu}=C_3\L_3(t)e^{-P^3|x|}\quad\mbox{where}\,\,C_3\triangleq C_0^{p^3}\left[\eta(P)^{p^2}\eta(P^2)^p\eta(P^3)\right],\quad\L_3(t)=\int_0^t\L(\t)\L_2(\t)^pd\t.
\end{align*}

{\it Step 2.} By induction, we obtain that
\begin{align}\label{Tmp:muCmLme}
\mu\geq C_{m}\L_m(t)e^{-P^m|x|}\quad(m\geq1),
\end{align}
where
\begin{align}
C_{m}&=C_0^{p^{m}}\prod_{j=1}^{m}\eta\left(P^j\right)^{p^{m-j}},
\quad\L_m(t)=\int_0^t\L(\t)\L_{m-1}(\t)^pd\t.
\end{align}
We have to calculate $\L_m=\L_m(t)$. By differentiation, we observe that each $\L_m$ solves 
\begin{align*}
\L_m'=\L\L_{m-1}^p,\quad\L_m(0)=0.
\end{align*}
Observe that $\L=\L_1'$. 
For $m=2$, we get 
\[
\L_2=\L_1'\L_1^p=\left(\frac{\L_1^{1+p}}{1+p}\right)'\quad\Rightarrow\quad\L_2=\frac{\L_1^{1+p}}{1+p}.
\]
For $m=3$, we get that
\[
\L_3'=\L_1'\L_2^p=\L_1'\frac{\L_1^{p+p^2}}{(1+p)^p}=\left(\frac{\L_1^{1+p+p^2}}{(1+p)^p(1+p+p^2)}\right)'\quad\Rightarrow\quad \L_3=\frac{\L_1^{1+p+p^2}}{(1+p)^p(1+p+p^2)}.
\]
By the same reasoning, it follows that
\begin{align}
\L_m=\frac{\L_\ast^{K_m}}{L_m},\quad K_m=1+p+\cdots+p^{m-1},\,\, L_m=\prod_{k=1}^{m-1}(1+p+\cdots+p^k)^{p^{m-1-k}}
\end{align}
Since $0<p<1$ and $1+p+\cdots+p^k\leq q$, we have $L_m\leq q^{K_{m-1}}$. Thus
\[
\L_m(t)\geq\frac{\L_\ast^q}{q^{K_{m-1}}}=(1-p)^{K_{m-1}}\L_\ast^q.
\]

Next we estimate $C_m$. Since $0<p,P<1$ and $\eta$ is non-increasing, we have
\begin{align*}
\begin{cases}
p^m\to0\,\,\mbox{as $m\to\I$},\\
\displaystyle 1+p+\cdots+p^k\leq q,\,\,\mbox{and}\\
\displaystyle\eta(P^j)\geq\eta(P)\,\,\mbox{for all $j\geq1$},
\end{cases}
\end{align*}
which implies
\begin{align*}
C_{m}&\geq C_0^{p^{m}}\prod_{j=1}^m\eta(P)^{p^{m-j}}=C_0^{p^{m}}\eta(P)^{K_m}.
\end{align*}
Now we obtain
\begin{align*}
\mu(x,t)\geq C_0^{p^{m}}\eta(P)^{K_m}(1-p)^{K_{m-1}}\L_\ast(t)^qe^{-P^{m}|x|}.
\end{align*}
By taking $m\to\I$ we obtain
\begin{align}\label{Tmp:muAlm}
\mu(x,t)\geq \left(\eta(P)(1-p)\L_\ast(t)\right)^q\quad(\mbox{on $Q_{\t_0}$}).
\end{align}

{\it Step 3.} We perform a further iteration. This time we use (\ref{Est:BVd}) that $\CB(V|x|^d)\geq\V_0\L(t)|x|^{\S+d}$ for all $d\geq0$. Also note that $pq=q-1$ and $\L=\L_\ast'$. Let us write (\ref{Tmp:muAlm}) as $\mu\geq\CI_0\triangleq M\L_\ast(t)^q$.  We have
\begin{align*}
\mu&\geq\CN^s\CI_0\geq M^p\int_0^t\CB\left(V\right)\L_\ast(\t)^{pq}d\t,\nonumber\\
&\geq\V_0 M^p|x|^\S\int_0^t\L(\t)\L_\ast(\t)^{q-1}d\t,\nonumber\\
&=\V_0(1-p)M^p|x|^\S\L_\ast(t)^q\triangleq \CI_1,
\end{align*}
and
\begin{align*}
\mu&\geq\CN^s\CI_1\geq\left(\V_0(1-p)\right)^pM^{p^2}\int_0^t\CB\left(V|\cdot|^{\S p}\right)\L_\ast(\t)^{pq}d\t,\nonumber\\
&\geq\left(\V_0(1-p)\right)^pM^{p^2}(\V_0|x|^{\S+\S p})\int_0^t\L(\t)\L_\ast(\t)^{q-1}d\t,\nonumber\\
&=\left(\V_0(1-p)|x|^\S\right)^{1+p}M^{p^2}\L_\ast(t)^q\triangleq \CI_2,\\
\mu&\geq\CN^s\CI_2\geq\left(\V_0(1-p)\right)^{p+p^2}M^{p^3}\int_0^t\CB\left(V|\cdot|^{\S p+\S p^2}\right)\L_\ast(\t)^{pq}d\t,\nonumber\\
&\hphantom{\mu}\geq\left(\V_0\right)^{p+p^2}M^{p^3}(\V_0|x|^{\S+\S p+\S p^2})\int_0^t\L_\ast(\t)^{q-1}d\t,\nonumber\\
&\hphantom{\mu}=\left(\V_0(1-p)|x|^\S\right)^{1+p+p^2}M^{p^3}\L_\ast(t)^q\triangleq\CI_3.
\end{align*}
By induction, we have
\begin{align*}
\mu(x,t)\geq\CI_m\triangleq\left(\V_0(1-p)|x|^\S\right)^{\sum_{j=0}^{m-1}p^j}M^{p^{m}}\L_\ast(t)^q\quad(m\geq1).
\end{align*}
Since $0<p<1$, by taking $m\to\I$, we obtain that 
\begin{align}
\mu(x,t)\geq\left(\V_0(1-p)|x|^\S\right)^q\L_\ast(t)^q=\left(\V_0(1-p)|x|^\S\L_\ast(t)\right)^q,
\end{align}
which proves Claim \ref{Claim:Lower}.\QED
\end{proof}

{\it Step 4.} We finish the proof of this theorem in the case that $u_0\gtrsim e^{-\D|x|}$. Setting $u=e^{-t}\mu$ in Claim \ref{Claim:Lower}, we have 
\begin{align}
u\geq \CI_0=C_0e^{-t}|x|^{\S q}\L_\ast(t)^q,\quad C_0\triangleq \left(\V_0(1-p)\right)^q.
\end{align}
Since $u$ is a s-supersolution, it follows, by Lemma \ref{Lem:TMisM}, that $u\geq\CM u$. In particular, $u$ satisfies
\begin{align}
u\geq\CN u\triangleq\int_0^t\CG(t-\t)\CB\left(Vu^p\right)(\t)\,d\t.
\end{align}
Employing Lemma \ref{Lem:BasicBG} (3) that $\CG(t)(|x|^d)\geq|x|^d$ for all $d\in J_n^+$ and (\ref{Est:BVd}) in the preceding inequality, we get
\begin{align*}
&u\geq\CN \CI_0\geq C_0^p\int_0^te^{-p\t}\CG(t-\t)\CB(V|\cdot|^{\S pq})\L_\ast(\t)^{pq}d\t,\nonumber\\
&\hphantom{u}\geq C_0^p\V_0\Vt e^{-pt}|x|^{\S+\S pq}\int_0^t\L(\t)\L_\ast(\t)^{q-1}d\t,\nonumber\\
&\hphantom{u}=C_1e^{-pt}|x|^{\S q}\L_\ast(t)^q\triangleq \CI_1,\,\,\mbox{where}\\
&C_1=C_0^p\V_0\Vt(1-p)=C_0.
\end{align*}
Similarly, we have
\begin{align*}
&u\geq\CN \CI_1\geq C_1^p\int_0^te^{-p^2\t}\CG(t-\t)\CB\left(V|\cdot|^{\S pq}\right)\L_\ast(\t)^{pq}d\t,\nonumber\\
&\hphantom{u}\geq C_1^p\V_0\Vt e^{-p^2t}|x|^{\S+\S pq}\int_0^t\L(\t)\L_\ast(\t)^{q-1}d\t,\nonumber\\
&\hphantom{u}=C_2e^{-p^2t}|x|^{\S q}\L_\ast(t)^q\triangleq \CI_2,\,\,\mbox{where}\\
&C_2=C_1^p\V_0\Vt(1-p)=C_0.
\end{align*}
It follows by induction that
\begin{align*}
u(x,t)\geq \CI_m\triangleq C_0e^{-p^mt}|x|^{\S q}\L_\ast(t)^q,\quad(m\geq0).
\end{align*}
Sending $m\to\I$ we obtain
\begin{align}
u(x,t)\geq C_0|x|^{\S q}\L_\ast(t)^q=\left(\V_0\Vt(1-p)|x|^\S\L_\ast(t)\right)^q.
\end{align}
This is the desired result of the theorem in the case $u_0\geq C_0e^{-\D|x|}$.\smallskip

{\it Step 5.} Finally, we prove the theorem in the general case that $u_0\geq0$ and $u_0\neq0$. Let $0<\t<\t_0$ and $\D>1$. By Lemma \ref{Lem:lowert0}, there is a constant $C_0=C_0(\D,\t,u_0)>0$ such that 
\begin{align}
\mu(\t)\geq\CM^s\mu(\t)\geq C_0e^{-\D|x|}\quad(\mbox{on $\R^n$}).
\end{align}
Using the semigroup property for s-supersolutions (Remark \ref{Rem:MsSem}), we have that $\mu^\t$ satisfies
\begin{align*}
&\mu^\t(t)\geq\CM^s_{(U_0,\T V)}\mu^\t(\t),\quad\mbox{where}\\
&U_0=\CM^s\mu(\t),\quad\T V(x,t)=e^{(1-p)\t}V(x,t+\t).
\end{align*}
By (\ref{HypVLp}), it follows that 
\begin{align*}
\T V(x,t)\geq\L(t+\t)|x|^\S\quad(|x|\geq1,0<t<\t_0-\t),
\end{align*}
Thus $U\triangleq e^{-t}\mu^\t$ is an s-supersolution of (\ref{Eqn:Main}) satisfying (\ref{HypVL}) with $\L(t)$ replaced by $\L(t+\t)$. Since $U_0=\CM^s\mu(\t)\geq C_0e^{-\D|x|}$, it follows by the previous special case that
\begin{align*}
u(x,t+\t)&=e^{-\t}U(x,t),\nonumber\\
&\geq e^{-\t}\left(\V_0(1-p)|x|^\S\int_0^t\L(s+\t)ds\right)^q\quad(\mbox{on $Q_{\t_0-\t}$}).
\end{align*}
This is true for all $0<\t<\t_0$. For each $0<t<\t_0$ and any $0<\t<t$, we have found that
\begin{align*}
u(x,t)&=u(x,(t-\t)+\t),\nonumber\\
&\geq e^{-\t}\left(\V_0(1-p)|x|^\S\int_0^{t-\t}\L(s+\t)ds\right)^q
\end{align*}
for all $x\in\R^n$. Letting $\t\to0$, $u$ satisfies the desired estimate.\QED
\end{proof}

\begin{remark}

A particularly important example of $\L$ is 
\[
\L(t)\sim t^k,\,\,(\log t)^\nu,\,\, t^k(\log t)^\nu\quad(k,\nu\in\R)\,\,\mbox{as $t\to\I$}.
\]
More generally, we shall consider $\L$ satisfying an \textit{$\V$-asymptotic bound condition}: $\L\in\CE_\V$ ($0<\V<1$) where
\begin{align}\tag{$\V$-AB}\label{Hyp:epsilonB}
\CE_\V=\left\{\L\in C(\R;\R_{>0}):\exists\,c,t_0>0\,\,\mbox{such that}\,\,\L_\ast(\V t)\geq c\L_\ast(t),\,\,\forall\,t\geq\V^{-1}t_0\right\}.
\end{align}

\begin{example}[\textbf{Power functions}]
Consider $\L>0$ satisfying 
\[
\L(t)\sim t^k\,\,(k\in\R)\quad\mbox{as $t\to\I$},
\]
i.e.\ there exist $t_0,\A_1,\A_2>0$ such that
\[
\A_1t^k\leq\L(t)\leq\A_2t^k\quad\mbox{for all $t\geq t_0$}.
\]
We claim that $\L\in\CE_\V$ for all $\V\in(0,1)$.\smallskip

We have
\begin{align}\label{Est:Lastnu}
\L_\ast(t)&=\int_0^{t_0}\L(s)ds+\int_{t_0}^t\L(s)ds
\underset{t\to\I}{\sim}
\begin{cases}
\displaystyle
1-t^{k+1}&\mbox{if $k<-1$},\\
\vspace{-10pt}\\
\displaystyle 1+\ln(t)&\mbox{if $k=-1$},\\
\vspace{-10pt}\\
\displaystyle 1+t^{k+1}&\mbox{if $k>-1$},
\end{cases}\nonumber\\
&
\underset{t\to\I}{\sim}\L_{\ast,k}(t)\triangleq
\begin{cases}
t^{(k+1)_+}&\mbox{if $k\neq-1$},\\
\ln t&\mbox{if $k=-1$}.
\end{cases}
\end{align}
The last asymptotic estimate means there exists $t_1>0$ such that $\L_\ast(t)\sim\L_{\ast,k}(t)$ for all $t\geq t_1$. 
For any $\V\in(0,1)$, if $\V t\geq t_1$ then 
\[
\L_{\ast}(\V t)\gtrsim\L_{\ast,k}(\V t)\gtrsim\L_{\ast,k}(t)\gtrsim\L_\ast(t).
\]
Therefore $\L\in\CE_\V$ as claim.
\end{example}

\begin{example}[\textbf{Logarithmic functions}]

Consider $\L>0$ satisfying 
\[
\L(t)\sim(\log t)^\nu\,\,(\nu\in\R)\quad\mbox{as $t\to\I$},
\]
i.e\ $\exists$ $\A_1,\A_2>0$ and $t_0>1$ such that 
\[
\A_1(\log t)^\nu\leq\L(t)\leq\A_2(\log t)^\nu\quad\mbox{for all $t\geq t_0$}.
\]
We show that $\L\in\CE_\V$ for any $\V\in(0,1)$.\smallskip

First assume $\nu\geq0$. We have
\begin{align*}
&\L_\ast(t)\lesssim A+\int_{t_0}^t(\log\t)^\nu\,d\t\underset{t\to\I}{\lesssim}t(\log t)^\nu,\\
&\L_\ast(t)\gtrsim A+t\int_{t_0/t}^1(\log(t\t))^\nu\,d\t\underset{t\to\I}{\gtrsim}t\int_{1/2}^1(\log(t/2))^\nu\,d\t\underset{t\to\I}{\gtrsim}t(\log t)^\nu.
\end{align*}
Then 
\[
\L_\ast(\V t)\gtrsim(\V t)(\log(\V t))^\nu\gtrsim t(\log t)^\nu\gtrsim\L_{\ast}(t),
\]
which means $\L\in\CE_\V$.\smallskip

Now assume $\nu<0$. Integrating by parts, we have
\begin{align*}
\int_{t_0}^t(\log\t)^\nu\,d\t&=A+t(\log t)^\nu+|\nu|\int_{t_0}^t(\log \t)^{\nu-1}d\t\leq A+t(\log t)^\nu+\frac{|\nu|}{\log t_0}\int_{t_0}^t(\log\t)^\nu d\t,
\end{align*}
Hence by taking $t_0$ sufficiently large, we have
\[
\int_{t_0}^t(\log\t)^\nu\,d\t\underset{t\to\I}{\lesssim}t(\log t)^\nu.
\]
On the other hand, it is clear that
\begin{align*}
\L_\ast(t)\geq A+t(\log\t)^\nu\underset{t\to\I}{\gtrsim}t(\log t)^\nu.
\end{align*}
Thus $\L_\ast(t)\underset{t\to\I}{\sim}t(\log t)^\nu$. It follows as above that $\L\in\CE_\V$. We note that the case $\L(t)=(\log t)^{-1}$ leads to the famous \textit{logarithmic integral function} $\mbox{li}\,(t)=\L_\ast(t)$ whose asymptotic property is well-known.

\end{example}

\begin{example}[\textbf{Mixed}]

Generally, consider a continuous function $\L>0$ satisfying
\begin{align*}
\L(t)\sim t^k(\log t)^\nu\quad(k,\nu\in\R)\quad\mbox{as $t\to\I$}.
\end{align*}
If $k=-1$ and $\nu\in\R$, then by direct integration, it follows that
\begin{align*}
\displaystyle\L_\ast(t)\sim A+\int_{t_0}^t\frac{(\log\t)^\nu}{\t}d\t\underset{t\to\I}{\sim}\begin{cases}
\log(\log t)&\mbox{if $\nu=-1$},\\
\vspace{-10pt}\\
\displaystyle(\log t)^{(\nu+1)_+}&\mbox{if $\nu\neq-1$}.
\end{cases}
\end{align*}

Now assume $k\neq-1$. If $\nu\geq0$ then we have
\begin{align*}
&\L_\ast(t)=A+\int_{t_0}^t\t^k(\log\t)^\nu d\t\leq A+(\log t)^\nu\frac{t^{k+1}}{k+1}\underset{t\to\I}{\lesssim}
\begin{cases}
\displaystyle 1&\mbox{if $k<-1$},\\
\vspace{-10pt}\\
\displaystyle t^{k+1}(\log t)^\nu&\mbox{if $k>-1$},
\end{cases}\\
&\L_\ast(t)\underset{t\to\I}{\gtrsim}A+\int_{t/2}^t\t^k(\log\t)^\nu d\t=A+t^{k+1}\int_{1/2}^1\t^k(\log(t/2))^\nu d\t\underset{t\to\I}{\gtrsim}A+t^{k+1}(\log t)^\nu.
\end{align*}
So we have
\begin{align*}
\L_\ast(t)\underset{t\to\I}{\sim}\left(t(\log t)^{\frac{\nu}{k+1}}\right)^{(k+1)_+}.
\end{align*}

Finally, consider the case $\nu<0$. Clearly, we have
\begin{align*}
\L_\ast(t)=A+\int_{t_0}^t\t^k(\log\t)^\nu d\t\geq A+(\log t)^\nu\int_{t_0}^t\t^kd\t\underset{t\to\I}{\gtrsim}\left(t(\log t)^{\frac{\nu}{k+1}}\right)^{(k+1)_+}
\end{align*}
Employing the integration by parts we get
\begin{align*}
\L_\ast(t)&=A+(k+1)^{-1}t^{k+1}(\log t)^\nu-\frac{\nu}{k+1}\int_{t_0}^t\t^k(\log\t)^{\nu-1}d\t,\\
&\leq A+(k+1)^{-1}t^{k+1}(\log t)^\nu+\frac{|\nu|}{(|k|+1)\log t_0}\int_{t_0}^t\t^k(\log\t)^{\nu}d\t
\end{align*}
Hence by choosing $t_0$ sufficiently large, we obtain
\[
\L_\ast(t)\underset{t\to\I}{\lesssim}\left(t(\log t)^{\frac{\nu}{k+1}}\right)^{(k+1)_+}.
\]

In summary we have shown that
\begin{align}
\L_\ast(t)\underset{t\to\I}{\sim}\L_{\ast,k,\nu}\triangleq
\begin{cases}
\displaystyle\log(\log t)&\mbox{if $k=\nu=-1$},\\
\vspace{-10pt}\\
\displaystyle(\log t)^{(\nu+1)_+}&\mbox{if $k=-1,\nu\neq-1$},\\
\vspace{-10pt}\\
\left(t(\log t)^{\frac{\nu}{k+1}}\right)^{(k+1)_+}&\mbox{all other cases}.
\end{cases}
\end{align}
It is now easy to see that $\L\in\CE_\V$ for all $k,\nu\in\R$.
\end{example}

\end{remark}

We can derive the following asymptotic lower grow-up rate for s-supersolutions.

\begin{theorem}[Lower asymptotic]\label{Thm:SharpLower}

Assume (\ref{HypVL}), (\ref{HypAL}), and $u_0\geq0$, $u_0\neq0$. In case $\S>0$, assume in addition that $\L\in\CE_\V$ for some $0<\V<1$. If $u\geq0$ is a global s-solution of (\ref{Eqn:Main}), then
\begin{align}\label{Sharp:1}
&u(x,t)\underset{t\to\I}{\gtrsim}\L_\ast(t)^{\frac{1}{1-p}}\left(1+t+|x|^2\right)^{\frac{\S}{2(1-p)}}\quad(\mbox{on $\R^n$}),\\
&\|u(\cdot,t)\|_{R,a}\underset{t\to\I}{\gtrsim}\L_\ast(t)^{\frac{1}{1-p}}t^{\frac{\S}{2(1-p)}}.\label{Est:SharpLower}
\end{align}

In the case $\L(t)\sim t^k(\log t)^\nu$ ($k,\nu\in\R$) as $t\to\I$, then
\begin{align}
&u(x,t)\underset{t\to\I}{\gtrsim}\L_{\ast,k,\nu}^{\frac{1}{1-p}}(1+t+|x|^2)^{\frac{\S}{2(1-p)}}\quad(\mbox{on $\R^n$}),\\
&\|u(\cdot,t)\|_{R,a}\underset{t\to\I} {\gtrsim}
\begin{cases}
\displaystyle(\log(\log t))^{\frac{1}{1-p}}t^{\frac{\S}{2(1-p)}}&\mbox{if $k=\nu=-1$},\\
\vspace{-10pt}\\
\displaystyle (\log t)^{\frac{(\nu+1)}{1-p}}t^{\frac{\S}{2(1-p)}}&\mbox{if $k=-1,\nu\neq-1$},\\
\vspace{-10pt}\\
\displaystyle
t^{\frac{\S+2(k+1)_+}{2(1-p)}}(\log t)^{\frac{\D\nu}{1-p}}&\mbox{if $k\neq-1$, $\nu\neq-1$, where $\D=\sign(k+1)$}.
\end{cases}
\end{align}

\end{theorem}

\begin{proof}

The second assertion follows from the first one and the remark above. So we only have to prove the first assertion. 
Since $u$ is an s-solution, it is a solution (Corollary \ref{Cor:PimM}). Then by the semigroup property (Lemma \ref{Lem:supersub}) and that $u^\ast(\t)=\CM u(\t)=u(\t)$, it follows that $u$ satisfies
\begin{align*}
u(x,\kappa+\t)\geq\CG(\t)u(x,\kappa)\quad(\forall\,\kappa,\t\geq0).
\end{align*}
Let $q=\frac{1}{1-p}$, $a_0=\frac{\S}{1-p}\geq0$. If $\S=0$ then $a_0=0$ and by Theorem \ref{Thm:uniflower} we obtain
\begin{align*}
&u(x,t)\geq C_0\L_\ast(t)^q,\quad C_0\triangleq\left(\V_0(1-p)\right)^q,\\
&\|u(\cdot,t)\|_{R,a}=\sup_x(R^2+|x|^2)^{-\frac{a}{2}}u(x,t)\geq C_0R^{-a}\L_\ast(t)^q,
\end{align*}
which are the desired estimates when $\S=0$. We note that, in this case, the result is true for any continuous function $\L>0$.\smallskip

Now we assume $\S>0$. By Theorem \ref{Thm:uniflower}, we have
\begin{align*}
u(x,t)\geq C_0\L_\ast(t)^q|x|^{a_0}.
\end{align*}
Since $|x|^{a_0}\geq(1/2)(1+|x|^2)^{a_0/2}$ for $|x|\geq r_0\triangleq(4^{1/a_0}-1)^{-1}$, we have
\begin{align*}
&u(x,t)\geq\frac{C_0}{2}\L_\ast(t)^q(1+|x|^2)^{a_0/2}\quad(|x|\geq r_0,t>0);\,\,\mbox{so}\\
&u(x,t)\geq\frac{C_0}{2}\L_\ast(t)^q\left((1+|x|^2)^{\frac{a_0}{2}}-(1+r_0^2)^{\frac{a_0}{2}}\right)\quad(x\in\R^n,t>0).
\end{align*}
Then by Proposition \ref{Prop:TwoSB} (2), there is $k\in(0,1)$ such that
\begin{align*}
\CG(\t)u(x,\kappa)&\geq \frac{C_0}{2}\L_\ast(\kappa)^q\left((1+k\t+|x|^2)^{\frac{a_0}{2}}-(1+r_0^2)^{\frac{a_0}{2}}\right),\\
&\geq\frac{C_0}{4}\L_\ast(\kappa)^q(1+k\t+|x|^2)^{\frac{a_0}{2}},
\end{align*}
provided $(1+k\t+|x|^2)^{a_0/2}\geq2(1+r_0^2)^{a_0/2}$ which is true by taking $\t\geq t_0\triangleq k^{-1}(4^{1/a_0}(1+r_0^2)-1)^{1/2}$. \smallskip

By assumption, there exist $\V\in(0,1)$ and $t_1,c>0$ such that $\L_\ast(\V t)\geq c\L_\ast(t)$ for all $t\geq \V^{-1}t_1$. Thus for $t\geq\max\{(1-\V)^{-1}t_0,\V^{-1}t_1\}$ we have 
\begin{align*}
u(x,t)&\geq\CG(t-\V t)u(x,\V t)\geq\frac{C_0}{4}\L_\ast(\V t)^q(1+k(t-\V t)+|x|^2)^{\frac{a_0}{2}},\\
&\geq \frac{C_0}{4}c^q\L_\ast(t)^q\left(1+k(1-\V)^{-1}t+|x|^2\right)^{\frac{a_0}{2}},\\
&\geq\frac{C_0}{4}c^q(\min\{1,k(1-\V)^{-1}\})^{\frac{a_0}{2}}\L_\ast(t)^q(1+t+|x|^2)^{\frac{a_0}{2}}.
\end{align*}
This gives the desired estimate (\ref{Sharp:1}).\smallskip

For (\ref{Est:SharpLower}), we have
\begin{align}
\|u(\cdot,t)\|_{R,a}&=\sup_{x}(R^2+|x|^2)^{-\frac{a}{2}}u(x,t),\nonumber\\
&\geq R^{-\frac{a_0}{2}}u(0,t)\gtrsim \L_\ast(t)^qt^{\frac{a_0}{2}}
\end{align}
as $t\to\I$, which is the desired estimate.\QED

\end{proof}

\begin{remark}\label{Rem:Grow}

We have discovered that the unbounded, non-autonomous, sublinear source $V(x,t)u^p\gtrsim \L(t)|x|^\S$ where $\L>0$ satisfies a certain condition and $\S\in J_n^+=\{0\}\cup[(2-n)_+,\I)$ has induced, on any nontrivial solution $u$ of Eq.\ (\ref{Eqn:Main}),
\begin{enumerate}
\item[(i)] a grow-up in time at least as
\[
\L_\ast(t)^{\frac{1}{1-p}}t^{\frac{\S}{2(1-p)}},
\]
where $\L_\ast(t)=\int_{0}^t\L(\t)d\t$ is an anti-derivative of $\L$, and 
\item[(ii)] a spatial growth at least as
\[
|x|^{\frac{\S}{1-p}}.
\]
\end{enumerate}
As has already been observed in \cite{Khomrutai15} when $\L\equiv1$ and $\S=0$ that the grow-up in time, $\L_\ast(t)^{1/(1-p)}=t^{1/(1-p)}$, is affected by the sublinearity. So the spatial growth of the potential $V$, combined with the sublinearity, induce a spatial growth of the solution at least as $|x|^{\S/(1-p)}$ and an additional grow-up (in time) factor $t^{\S/(2(1-p))}$. We note that the solution can exhibit the growth as $|x|\to\I$, at any later time, even when its initial state is a bounded or decaying function, displaying a {\it regularization effect}.

\end{remark}

\section{Comparison theorem; uniqueness of solutions}\label{Sec:CompU}

We prove the key comparison theorem for Eq.\ (\ref{Eqn:Main}) assuming that $V$ satisfies (\ref{Hyp:Vcomp}), that is
\begin{align}\label{HypV}
\begin{cases}
\displaystyle V(x,t)\geq\Uv\L(t)|x|^\S&\mbox{for $|x|\geq1,0<t<\t_0$},\\
\vspace{-10pt}\\
\displaystyle V(x,t)\leq\Ov\L(t)|x|^\S&\mbox{for $x\in\R^n,0<t<\t_0$},
\end{cases}
\end{align}
where $\L\in C(\R;\R_{>0})$, $0<\t_0\leq\I$, and $\S\in J_n^+$, for some constants $\Uv,\Ov>0$. Here it is no loss of generality in assuming ``$|x|\geq1$" instead of ``$|x|\geq R_0$" for some $R_0>0$.

\begin{theorem}[Comparison theorem]\label{Thm:Comp}

Assume (\ref{HypAL}) and (\ref{Hyp:Vcomp}) (see (\ref{HypV})). Let $0\leq u,v\in\CZ_{\t_0,a}$ be s-supersolution and s-subsolution, respectively, of (\ref{Eqn:Main}), i.e. $\mu=e^{t}u$ and $\omg=e^{t}v$ satisfy
\begin{align}
\begin{cases}
\displaystyle
\mu(t)\geq u_0+\int_0^t\CB\mu(\t)d\t+\int_0^te^{(1-p)\t}\CB\left(V\mu^p\right)(\t)d\t,\\
\vspace{-10pt}\\
\displaystyle\omg(t)\leq v_0+\int_0^t\CB\omg(\t)d\t+\int_0^te^{(1-p)\t}\CB\left(V\omg^p\right)(\t)d\t,
\end{cases}
\end{align}
on $Q_{\t_0}$ where $0\leq u_0,v_0\in BC_a$. If $p$ satisfies (\ref{HypCompOnp}) 
where $\V_0>0$ is the constant given in (\ref{Est:BVd}), and
\begin{align}
u_0\geq v_0\geq0\quad\mbox{with}\quad u_0\neq0,
\end{align}
then $u\geq v$ on $Q_{\t_0}$.

\end{theorem}

\begin{proof}
We will prove the result when $\t_0<\I$. The case $\t_0=\I$ then follows by considering the problem on any finite subintervals.\smallskip

Let $0<\th<1$ to be specified and $R=R(n,a,\th)>0$ be sufficiently large according to Proposition \ref{Prop:BoundBG} (2). Recall $\LL=(R^2+|x|^2)^{-a/2}$ so that $\|\vp\|_{R,a}=\|\LL\vp\|_{L^\I}$. We denote $\|\cdot\|=\|\cdot\|_{R,a}$.\smallskip

Let $q=1/(1-p)$. We define
\begin{align}
w=(v-u)_+.
\end{align}
By Lemma \ref{Lem:TMisM}, $u,v$ are supersolution and subsolution, respectively, of (\ref{Eqn:Main}), and since $v_0\leq u_0$, we have
\begin{align}
(v-u)(t)&\leq \CG(t)(v_0-u_0)+\int_0^t\CG(t-\t)\CB\left[V(v^p-u^p)\right](\t)\,d\t,\nonumber\\
&\leq\int_0^t\CG(t-\t)\CB\left[V(v-u)_+^p\right](\t)\,d\t\quad(\mbox{by Lemma \ref{Technical1} (1)}).
\end{align}
This implies that $w\geq0$ satisfies the integral inequality
\begin{align}
w(t)\leq\int_0^t\CG(t-\t)\CB\left[Vw^p\right](\t)\,d\t.
\end{align}
Since $V(x,\t)\leq\Ov\L(\t)|x|^\S$, one can prove as in Lemma \ref{Lem:LamV} to find that 
\[
\|\CG(t-\t)\CB(Vw^p)(\t)\|\leq\th^{-1}\Ov\L(\t)\B(t-\t)\|w(\t)\|^p.
\]

Let $0<T<\t_0$ to be specified. For $t\in(0,T]$, by taking the norm, we get
\begin{align}
\|w(t)\|&\leq\int_0^t\left\|\CG(t-\t)\CB\left(Vw^p\right)(\t)\right\|d\t,\nonumber\\
&\leq\th^{-1}\Ov\int_0^t\L(\t)\B(t-\t)\|w(\t)\|^pd\t,
\nonumber\\
&\leq\th^{-1}\Ov\B(T)\int_0^t\L(\t)\|w(\t)\|^pd\t.
\end{align}
Then we have by Lemma \ref{Technical1} (2) that
\begin{align}
\|w(t)\|\leq
\vU_T\L_\ast(t)^q,\quad\vU_T\triangleq\left(\th^{-1}\Ov\B(T)(1-p)\right)^q,\label{Tmp:Fin}
\end{align}
where $\L_\ast(t)=\int_0^t\L(\t)d\t$. 

\begin{claim}
At each $x\in\R^n,t>0$, we have
\begin{align}\label{Est:Uniquetmp1}
(v^p-u^p)_+\leq pq\left(\V_0\Uv|x|^\S\L_\ast(t)\right)^{-1}w.
\end{align}
\end{claim}

\begin{proof}[{\bf Claim}]
At each $(x,t)$ we have by the fundamental theorem of calculus that
\begin{align}
v^p-u^p&=\int_0^1\frac{d}{d\L}\left(\L v+(1-\L)u\right)^pd\L,\nonumber\\
&=p(v-u)\Theta^{p-1},
\end{align}
where $\Theta$ is a number between $v$ and $u$. The claim is trivial if $v\leq u$. On the other hand, if $v\geq u$ then using that $u_0\neq 0$ and $u$ is an s-supersolution of (\ref{Eqn:Main}), it follows by Theorem \ref{Thm:uniflower} that
\begin{align}
\Theta\geq u\geq\left(\V_0\Uv(1-p)|x|^\S\L_\ast(t)\right)^{\frac{1}{1-p}}.
\end{align}
Hence (\ref{Est:Uniquetmp1}) is also true in this case, therefore the claim is true.\QED
\end{proof}


We perform an alternative estimation for $w$. By the claim, we have
\begin{align}
&(v-u)(x,t)\leq\int_0^t\CG(t-\t)\CB\left[V(v^p-u^p)_+\right](x,\t)d\t,\nonumber\\
&\hphantom{(v-u)(x,t)}\leq pq\int_0^t\CG(t-\t)\CB\Big[V\Big(\V_0\Uv|x|^\S\L_\ast(\t)\Big)^{-1}w\Big]d\t,\nonumber\\
&\hphantom{(v-u)(x,t)}\leq\gamma pq\int_0^t\frac{\L(\t)}{\L_\ast(\t)}\CG(t-\t)\CB\left[w\right](x,\t)\,d\t,\\
&\mbox{where}\,\,\gamma\triangleq\V_0^{-1}(\Ov/\Uv)=\frac{1}{p_0},\nonumber
\end{align}
here we have used that $V(x,\t)\leq\Ov\L(\t)|x|^\S$. Now for $0<t\leq T$, we have by taking the norm that
\begin{align}
\|w(t)\|&\leq\gamma pq\theta^{-1}\int_0^t\frac{\L(\t)}{\L_\ast(\t)}\B(t-\t)\|w(\t)\|d\t,\nonumber\\
&\leq\VS_T\int_0^t\frac{\L(\t)}{\L_\ast(\t)}\|w(\t)\|d\t\quad(\VS_T\triangleq \gamma pq\theta^{-1}\B(T)).\label{Eqn:IntgEst}
\end{align}

Similar to \cite{AguirEscobedo86}, we introduce the function
\begin{align}
H(t)=\VS_T\int_0^t\frac{\L(\t)}{\L_\ast(\t)}\|w(\t)\|d\t.
\end{align}
Then (\ref{Eqn:IntgEst}) becomes
\begin{align}
H'(t)\leq\VS_T\frac{\L(t)}{\L_\ast(t)} H(t).
\end{align}
Let $0<\V<T$. Using that $\L(t)=\L_\ast'(t)$, we have for all $t\in[\V,T)$ that
\begin{align*}
\left(\log H(t)\right)'\leq\VS_T\left(\log\L_\ast(t)\right)'\quad\Rightarrow\quad H(t)\leq H(\V)\L_\ast(\V)^{-\VS_T}\L_\ast(t)^{\VS_T}.
\end{align*}
On the other hand, by (\ref{Tmp:Fin}) and the definition of $H$, we get that
\begin{align}
H(\V)&=\VS_T\int_0^\V\frac{\L(\t)}{\L_\ast(\t)}\|w(\t)\|d\t,\nonumber\\
&\leq\VS_T\vU_T\int_0^\V\L(\t)\L_\ast(\t)^{q-1}d\t,\nonumber\\
&=\VS_T\vU_T(1-p)\L_\ast(\V)^q.
\end{align}
Thus we obtain
\begin{align}\label{CompEps}
H(t)&\leq\VS_T\vU_T(1-p)\L_\ast(\V)^{q-\VS_T}\L_\ast(t)^{\VS_T}\quad(0<t<T).
\end{align}
The exponent of $\L_\ast(\V)$ in this inequality is
\[
q-\VS_T
=q\left(1-\gamma p\th^{-1}\B(T)\right)=q\left(1-\frac{p}{p_0}\th^{-1}\B(T)\right).
\]

Since $0<p<p_0$, we have $p/p_0\in(0,1)$. Now we choose $\th\in(0,1)$ such that $(p/p_0)\th^{-1}\in(0,1)$. Since $\B(T)\to 1$ as $T\to0$, we now choose $T>0$ sufficiently small such that $(p/p_0)\th^{-1}\B(T)<1$. It follows from the choice of $\th,T$ that
\begin{align}
q-\VS_T>0.
\end{align}
Observe that $T$ depends only on $p$ and $p_0$ and is independent of $\V$. Therefore passing $\V\to0$ in (\ref{CompEps}), we obtain $H(t)=0$ for all $0<t<T$. This implies in particular that
\begin{align}
v\leq u\qquad\mbox{on $Q_{T/2}$}.
\end{align}

Finally, we employ the semigroup argument, see Lemma \ref{Lem:supersub} and Remark \ref{Rem:MsSem}. By the preceding result, we have $u(T/2)\geq v(T/2)\geq0$ with $u(T/2)\not\equiv0$, hence $u^\dag(T/2)\geq v^\dag(T/2)\geq0$ with $u^\dag(T/2)\not\equiv0$ on $\R^n$. Recall $w^\dag(t)=\CM^sw(t)$. According to Remark \ref{Rem:MsSem}, setting $\T u=u^{T/2}$, $\T v=v^{T/2}$, the translations by time $T/2$, then $\T\mu\triangleq  e^t\T u,\T\omg\triangleq e^t\T v$ satisfy
\begin{align}
\begin{cases}
\displaystyle
\T\mu(t)\geq e^{-T/2}\mu^\dag(T/2)+\int_0^t\CB\T\mu(\t)d\t+\int_0^te^{(1-p)\t}\CB(V^{T/2}\T\mu^p)(\t)\,d\t,\\
\vspace{-10pt}\\
\displaystyle
\T\omg(t)\leq e^{-T/2}\omg^\dag(T/2)+\int_0^t\CB\T\omg(\t)d\t+\int_0^te^{(1-p)\t}\CB(V^{T/2}\T\omg^p)(\t)d\t.
\end{cases}
\end{align}
Since $V^{T/2}$ satisfies (\ref{HypV}) with the same ratio $p_0$ and $T$ depends only on $p$ and $p_0$, we obtain as above that $v^{T/2}\leq u^{T/2}$ on $Q_{T/2}$. Thus we have
\begin{align}
v\leq u\quad\mbox{on $Q_{T}$}.
\end{align}
Similarly, we can continue the argument to get that $v\leq u$ on $Q_{3T/2},Q_{2T},\ldots$ and so forth. Therefore $v\leq u$ on $Q_{\t_0}$ as desired.\QED
\end{proof}

\begin{remark}\label{Rem:SharpComp}
\begin{enumerate}
\item[(i)] Generally $\V_0\leq1$. We have $\V_0=1$ if the lower bound condition for the potential in (\ref{HypV}) is true on the entire $\R^n\times(0,\t_0)$. See Remark \ref{Rem:AftLow} (i).
\item[(ii)] \label{Rem:Opt}
The comparison theorem is optimal in the sense that one cannot expect the result to be true for 
\[
p\geq p_0=\V_0(\Uv/\Ov).
\]
This can be seen from the following well-studied problem of finding entire positive solutions for the sublinear elliptic equation
\begin{align}\label{Subell}
\Lap w+\r(x)w^p=0\quad\mbox{on $\R^n$},\quad\mbox{where}\,\,\r(x)\gneq0,0<p<1.
\end{align}
It was shown in \cite{BrezisKamin92} that this elliptic equation has a bounded solution $w>0$ if and only if there is a bounded solution $U$ for the equation $-\Lap U=\r(x)$. The latter means $\r$ must decay at a sufficiently fast rate, e.g.\ $\r(x)\sim|x|^{-m}$ at infinity, $m\geq2$. Moreover, it was shown in  \cite{BrezisKamin92,Dinu06} that (\ref{Subell}) has infinitely many bounded solutions, where for each constant $\ell\geq0$ there is a unique solution satisfying
\begin{align}
\lim_{|x|\to\I}w(x)=\ell.
\end{align}
See \cite{Murata95} for an important study on the necessary and sufficient conditions on the potential function $V(x)$ such that the uniqueness solution property holds for the linear Cauchy problem $\P_tu=\Lap u+V(x)u$ in $\R^n\times(0,T)$, $u(x,0)=u_0(x)$.\smallskip

Now fix $V=\r(x)$ having a sufficiently fast spatial decaying rate. Note that $\Uv/\Ov=0$. Choose two solutions $w_1,w_2$ of (\ref{Subell}) such that $2w_1\leq w_2\leq Aw_1$ for a sufficiently large constant $A>1$. Let $\psi(t)$ be the solution of 
\begin{align}
\psi'=\psi^p-\psi\quad\mbox{for $t>0$},\quad\psi(0)=A.
\end{align}
In fact, we have 
\[
\psi(t)=\left[(A^{1-p}-1)e^{-t(1-p)}+1\right]^{\frac{1}{1-p}}.
\]
Hence $\psi(t)\to1$ as $t\to\I$. Let $u_1=\psi w_1$ and $u_2=w_2$. It is easy to show that $u_1$ is a supersolution of (\ref{Eqn:Main}) with $u_1|_{t=0}=Aw_1\geq w_2=u_2|_{t=0}$ while $u_2$ is a solution of (\ref{Eqn:Main}). However, for all sufficiently large $t>0$ we have $u_1\leq2w_1\leq u_2$. Therefore the comparison result is not true when $V=\r(x)$ for all $0<p<1$.
%
\end{enumerate}

\end{remark}

\begin{corollary}\label{Cor:Comp1}
Assume as in Theorem \ref{Thm:Comp} but
\begin{align}\label{Hyp:New}
V\begin{cases}
=\displaystyle\L(t)|x|^\S&\mbox{if $|x|\geq1,0<t<\t_0$},\\
\vspace{-10pt}\\
\displaystyle\leq\L(t)|x|^\S&\mbox{if $|x|\leq1,0<t<\t_0$},
\end{cases}
\end{align}
where $\L(t)>0$ is a continuous function and $\S\in J_n^+$. If $0<p<\V_0^{-1}$ and $u_0\geq v_0$ with $u_0\neq0$, then
\[
u\geq v\quad\mbox{on $Q_{\t_0}$}.
\]

\end{corollary}

\begin{corollary}\label{Cor:Compare}
Assume as in Theorem \ref{Thm:Comp} but 
\begin{align}
V=\L(t)|x|^\S\quad(x\in\R^n,t>0),
\end{align}
where $\L(t)>0$ is a continuous function and $\S\in J_n^+$. If $0<p<1$ and $u_0\geq v_0$ with $u_0\neq0$, then
\[
u\geq v\quad\mbox{on $Q_{\t_0}$}.
\]
\end{corollary}

We conclude this section with the following existence and uniqueness result.

\begin{theorem}[Well-posedness]\label{Thm:Unique}

Assume (\ref{HypAL}), (\ref{Hyp:Vcomp}) (see \ref{HypV}), and $0\leq u_0\in BC_a$ with $u_0\neq0$.
\begin{enumerate}
\item[(1)] If $0<p<\V_0(\Uv/\Ov)$ then Eq.\ (\ref{Eqn:Main}) has a unique global s-solution 
\[
0\leq u\in \CZ_{\I,a}.
\]
\item[(2)] Assume further that $V=\L(t)|x|^\S$ where $\L(t)>0$ is a continuous function and $\S\in J_n^+$. If $0<p<1$ then Eq.\ (\ref{Eqn:Main}) has a unique global s-solution 
\[
0\leq u\in\CZ_{\I,a}.
\]
\end{enumerate}

\end{theorem}


\section{Non-uniqueness of solutions; classifications}\label{Sec:Zero}

In this section, we study the non-uniqueness of positive solutions for (\ref{Eqn:Main}) in the zero-initial case, Eq.\ (\ref{Eqn:zero}). We assume that $V,p$ and $u\in BC_a$ with $a\geq0$ satisfy (\ref{HypAL}), (\ref{Hyp:Vcomp}) (see \ref{HypV}), (\ref{HypCompOnp}) respectively, that is
\begin{align}\label{HypNonUnq}
\begin{cases}
\displaystyle V(x,t)\geq\Uv\L(t)|x|^\S\quad(|x|\geq1,t>0),\\
\vspace{-10pt}\\
\displaystyle V(x,t)\leq\Ov\L(t)|x|^\S\quad(x\in\R^n,t>0),\\
\vspace{-10pt}\\
0<p<p_0\triangleq\V_0(\Uv/\Ov),\\
\vspace{-10pt}\\
\displaystyle a\geq\frac{\S}{1-p},\quad\S\in J_n^+,
\end{cases}
\end{align}
where $\L(t)>0$ is a continuous function and $\S\in J_n^+,\Ov,\Uv>0$ are constants. Then, according to Theorem \ref{Thm:Existence}, the problem (\ref{Eqn:zero}) admits a global s-solution $0\leq u\in\CZ_{\I,a}$. Of course $u\equiv0$ is a solution. However, it will be shown that such a solution is not unique, in fact, the family of solutions is continuum.\smallskip

In the case $V=\L=\mbox{a constant}>0$, it was shown \cite{Khomrutai15} that all non-trivial s-solutions of the problem (\ref{Eqn:zero}) have the form
\begin{align}
u=u_\ast(\left[t-\kappa\right]_+)\quad\mbox{for some $\kappa\geq0$},
\end{align}
where 
\begin{align}
u_\ast(t)=\left((1-p)\L t\right)^{\frac{1}{1-p}}
\end{align}
is the maximal solution to (\ref{Eqn:zero}) with $V=\L$. Observe that $u_\ast$ is the unique solution to the problem
\begin{align*}
\begin{cases}
\displaystyle
\frac{d\vp}{dt}=\L\vp^p,\\
\vspace{-10pt}\\
\displaystyle\vp(0)=0,\quad\vp>0,
\end{cases}
\end{align*}
which is the problem of finding spatial independent solutions for (\ref{Eqn:Main}) in this case.
The function $u_\ast([\cdot-\kappa]_+)$ is simply the delay by time $\kappa$ of $u^\ast$. 
\smallskip

We prove in this section that a similar result holds when $V=\L(t)$, a positive continuous function. Then we generalize the result to the case when the potential is {\it time-independent}. The situation is much more complicated if $V$ is both $x$- and $t$-dependent. In the latter case, a one-side {\it convexity condition} has to be imposed on the potential. 
\smallskip

Let us introduce the following important definition.

\begin{definition}[Maximal solutions]

A function $0\leq u_\ast=e^{-t}\mu_\ast\in\CZ_{\I,a}$ is called a maximal s-solution of Eq.\ (\ref{Eqn:zero}) if it satisfies $\mu_\ast=\CM^s_{(0,V)}\mu_\ast$, where $\CM^s_{(0,V)}$ is defined by (\ref{Def:Ms}), and if $0\leq u=e^{-t}\mu\in\CZ_{\I,a}$ also satisfies $\mu=\CM^s_{(0,V)}\mu$ then 
\begin{align}
u\leq u_\ast\qquad\mbox{on $Q_\I$}.
\end{align}

\end{definition}

We have the following uniqueness of maximal s-solution result.

\begin{theorem}[Uniquness of maximal solutions]\label{Thm:UniqueMax}
Assume (\ref{HypNonUnq}). Then the problem (\ref{Eqn:zero}) admits a unique maximal s-solution.
\end{theorem}

\begin{proof}
We prove the existence. For each $m\in\BN$, let $0\leq u_m=e^{-t}\mu_m\in\CZ_{\I,a}$ be the unique global solution of the the problem $\mu_m=\CM^s_{(1/m,V)}\mu_m$ (Theorem \ref{Thm:Unique} (1)), that is $\mu_m$ satisfies
\begin{align}
\mu_m=\frac{1}{m}+\int_0^t\CB\mu_m(\t)d\t+\int_0^te^{(1-p)\t}\CB\left(V\mu_m^p\right)(\t)\,d\t.
\end{align}
For all $\ell>m$, we obtain by the comparison principle (Theorem \ref{Thm:Comp}) and the lower grow-up rate (Theorem \ref{Thm:uniflower}) that 
\begin{align}
u_m(t)\geq u_\ell(t)\geq\left(\V_0\Uv(1-p)|x|^\S\L_\ast(t)\right)^{\frac{1}{1-p}}\quad\mbox{(on $Q_\I$)}.
\end{align}
So $\{u_m\}_{m=1}^\I$ is pointwise non-increasing and bounded below, hence there is a function $u_\ast=e^{-t}\mu_\ast$ such that $u_m\to u_\ast$ pointwise as $m\to\I$. Moreover, $u_\ast$ is bounded below by
\begin{align}
u_\ast(t)\geq\left(\V_0\Uv(1-p)|x|^\S\L_\ast(t)\right)^{\frac{1}{1-p}}.
\end{align}
Applying the monotone convergence theorem as $m\to\I$, it follows that $\mu_\ast$ satisfies $\mu_\ast=\CM^s_{(0,V)}\mu_\ast$. Employing the same argument as in the proof of Theorem \ref{Thm:Existence}, one can show that $u_\ast\in\CZ_{\I,a}$.\smallskip

To show that $u_\ast$ is maximal, we assume that $0\leq u\in\CZ_{\I,a}$ is also an s-solution of (\ref{Eqn:zero}). By Theorem \ref{Thm:Comp} and the equation satisfied by $u_m$, then we have
\begin{align}
u(t)\leq u_m(t)\quad\mbox{(on $Q_\I$), $\forall\,m\geq1$.}
\end{align}
Since $u_m\to u_\ast$ as $m\to\I$, it follows that $u\leq u_\ast$ on $Q_\I$. So $u_\ast$ is a maximal s-solution of (\ref{Eqn:zero}). The uniqueness of $u_\ast$ is obvious.\QED
\end{proof}

\begin{notation}

For a given potential function $V$ satisfying (\ref{HypNonUnq}), the unique maximal s-solution of (\ref{Eqn:zero}) according to the preceding theorem will be denoted by $u_{\ast V}$. Observe that by the proof of the uniqueness of maximal solutions, we have
\begin{align}
u_{\ast V}(x,t)\geq\left(\V_0\Uv(1-p)|x|^\S\L_\ast(t)\right)^{\frac{1}{1-p}}\quad\mbox{on $Q_\I$}.
\end{align}

\end{notation}

The following result can proved easily using the semigroup property.

\begin{lemma}\label{Lem:Semizero}

Assume (\ref{HypNonUnq}). For $\kappa\geq0$, let $v=u^\kappa$ and $w=u([t-\kappa]_+)$ where $u\geq0$ is an s-solution of (\ref{Eqn:zero}) and $\kappa\geq0$. Then $v$ is an s-solution of the problem
\begin{align*}
\begin{cases}
\P_tv-\Lap\P_tv=\Lap v+V^\kappa(x,t)v^p& x\in\R^n,t>0,\\
\qquad v|_{t=0}=u(\kappa)&x\in\R^n,
\end{cases}
\end{align*}
and $w$ is an s-solution of the problem
\begin{align*}
\begin{cases}
\P_tw-\Lap\P_tw=\Lap w+V^{-\kappa}(x,t)w^p& x\in\R^n,t>0,\\
\qquad w|_{t=0}=0&x\in\R^n,
\end{cases}
\end{align*}
where $V^\kappa\triangleq V(x,t+\kappa)$ and $V^{-\kappa}\triangleq V(x,[t-\kappa]_+)$.

\end{lemma}

\begin{proof}
Let us prove the second assertion. We denote $\T V=V(x,[t-\kappa]_+)$. We can assume $\kappa>0$. Let 
\[
\omg(t)=e^tw(t)=e^tu([t-\kappa]_+)\quad\mbox{so}\quad\omg(t)=e^\kappa\mu([t-\kappa]_+).
\]
If $0\leq t\leq \kappa$ then $\omg(t)=0=\CM^s_{(0,\T V)}\omg(t)$ trivially. Let $t>\kappa$. Then we have
\begin{align}
\CM_{(0,\T V)}^s\omg(t)&=\int_\kappa^t\CB\omg(\t)d\t+\int_{\kappa}^te^{(1-p)\t}\CB\left(\T V\omg^p\right)(\t)\,d\t,\nonumber\\
&=\int_0^{t-\kappa}e^{\kappa}\CB\mu(\t)d\t+\int_0^{t-\kappa}e^{(1-p)(\t+\kappa)}e^{\kappa p}\CB\left(V\mu^p\right)(\t)\,d\t,\nonumber\\
&=e^\kappa\left(\int_0^{t-\kappa}\CB\mu(\t)d\t+\int_0^{t-\kappa}e^{(1-p)\t}\CB(V\mu^p)(\t)d\t\right),\nonumber\\
&=e^{\kappa}\mu(t-\kappa)=\omg(t).\label{NonUniqueCM}
\end{align}
Thus $w$ satisfies the second assertion.\QED
\end{proof}

Next we consider the case when $V=\L(t)$ where $\L>0$ is a continuous function. The following lemma is obvious so we omit the proof.

\begin{lemma}
Assume $\L(t)>0$ is a continuous function. The problem
\begin{align}
\begin{cases}
\displaystyle\frac{d}{dt}\vp=\L(t)\vp^p&t>0,\\
\vspace{-10pt}\\
\displaystyle \vp(0)=c\geq0,\quad\vp>0
\end{cases}
\end{align}
has the unique solution
\begin{align}
\vp(t)=\varPhi_\L(t;c)\triangleq\left(c^{1-p}+(1-p)\L_\ast(t)\right)^{\frac{1}{1-p}},\quad\L_\ast(t)=\int_0^t\L(\t)\,d\t.
\end{align}

\end{lemma}

\begin{remark}\label{Rem:VLt}

If $\L\equiv0$ on $[0,t_0]$ for some $t_0>0$ then $\vp\equiv c$ on $[0,t_0]$. The same conclusion holds for the pseudoparabolic problem:
\begin{align*}
\P_tu-\Lap\P_tu=\Lap u+\L(t)u^p,\quad u|_{t=0}=0,
\end{align*}
that is if $\L\equiv0$ on $[0,t_0]$ then $u\equiv c$ on $Q_{t_0}$. Therefore in the following analysis it is no loss of generality in assuming $\L(t)>0$ on some $[0,t_0]$, $t_0>0$.

\end{remark}

\begin{theorem}[Classification I]\label{Thm:Classf1}

Assume $V=\L(t)\geq0$ is a continuous function. A function $u\geq0$ is a nontrivial solution of the problem (\ref{Eqn:zero}) if and only if there is a constant $\kappa\geq0$ such that
\[
u=\left((1-p)(\L^\kappa)_\ast([t-\kappa]_+)\right)^{\frac{1}{1-p}}=\left((1-p)\int_0^{[t-\kappa]_+}\L(s+\kappa)ds\right)^{\frac{1}{1-p}}.
\]
\end{theorem}

\begin{proof}

The converse part can be proved easily using Lemma \ref{Lem:Semizero}. By Remark \ref{Rem:VLt}, we can assume $\L(t)>0$ on some $[0,t_0]$. This implies $\L_\ast(t)>0$ for all $t>0$ hence $u_{\ast\L}>0$ where $u_{\ast\L}$ is the maximal solution. Clearly we have
\[
u\leq u_{\ast\L}\quad\mbox{and}\quad\varPhi_\L(t;0)\leq u_{\ast\L}.
\]
On the other hand, for any $c>0$, we have by the comparison principle (Corollary \ref{Cor:Compare}) that
\[
\varPhi_\L(t;c)\geq u_{\ast\L}\,\,(\forall c>0)\quad\overset{c\to0}{\Rightarrow}\quad\varPhi_\L(t;0)\geq u_{\ast\L}.
\]
Thus we have
\begin{align*}
u_{\ast\L}(t)=\left((1-p)\L_\ast(t)\right)^{q}.
\end{align*}

Define
\begin{align}
\kappa=\inf\{t>0:u(x,t)>0\,\,\mbox{for some $x\in\R^n$}\}.\label{NonUniquet0}
\end{align}
Consider the case $\kappa=0$. Then in this case, there is a sequence $\{(x_m,t_m)\}_{m=1}^\I$ such that $t_m\to0^+$ and $u(x_m,t_m)>0$ for all $m$. This implies $u(x,t)>0$ for all $x\in\R^n,t>0$ by Lemma \ref{Lem:lowert0}. Let $\t>0$. By Lemma \ref{Lem:Semizero}, $u^\t=u(\cdot+\t)$ is an s-solution of the problem
\begin{align*}
\P_tu^\t-\Lap\P_tu^\t=\Lap u^\t+\L(t+\t)(u^\t)^p,\quad u^\t|_{t=0}>0.
\end{align*}
So we get by the lower grow-up result (Theorem \ref{Thm:uniflower}) that
\[
u(x,t+\t)=u^\t(x,t)\geq\left((1-p)\int_0^t\L(s+\t)ds\right)^{q}\quad\overset{\t\to0}{\Rightarrow}\quad u(x,t)\geq\left((1-p)\L_\ast(t)\right)^q.
\]
Therefore
\[
u=\left((1-p)\L_\ast(t)\right)^q.
\]

Now assume $\kappa>0$. Then by Lemma \ref{Lem:Semizero}, $v=u^\kappa$ satisfies
\begin{align*}
\P_tv-\Lap\P_tv=\Lap v+\L(t+\kappa)v^p,\quad v|_{t=0}=0
\end{align*}
and $v(t)>0$ for all $t>0$. We can conclude by the previous case that
\[
v(t)=\left((1-p)\int_0^t\L(s+\kappa)ds\right)^q\quad\overset{t\geq\kappa}{\Longrightarrow}\quad u(t)=\left((1-p)\int_0^{t-\kappa}\L^\kappa(s)ds\right)^q.
\]
Since $u\equiv0$ on $[0,\kappa]$, we obtain the formula of $u$ as desired.\QED

\end{proof}

Next we characterize non-trivial s-solutions for (\ref{Eqn:zero}) in the case $V=V(x)$.

\begin{theorem}[Classification II]\label{Thm:Classf2}

Assume (\ref{HypNonUnq}) and furthermore $V=V(x)$. A function $0\leq u\in\CZ_{\I,a}$ is a nontrivial s-solution of (\ref{Eqn:zero}) if and only if $u=u_{\ast V}\left(x,[t-\kappa]_+\right)$ for some constant $\kappa\geq0$.

\end{theorem}

\begin{proof} 
The converse part is true by Lemma \ref{Lem:Semizero} and the fact that $V$ is time-independent. \smallskip

Now we prove the direct part. Let $u$ be a nontrivial s-solution of (\ref{Eqn:zero}). Then we immediately have that $u\leq u_{\ast V}$. Define $\kappa$ by (\ref{NonUniquet0}).
Let us consider the case $\kappa=0$. Then we can argue as in the previous theorem to find that $u(x,t)>0$ for all $x\in\R^n,t>0$. Let $\t>0$. From the converse part and Lemma \ref{Lem:Semizero}, $w=u_{\ast V}([\cdot-2\t]_+)$ and $v=u^\t$ solve
\begin{align*}
\begin{cases}
\displaystyle\P_tw-\Lap\P_tw=\Lap w+V(x)w^p,\quad w|_{t=0}=0,\\
\vspace{-10pt}\\
\displaystyle\P_tv-\Lap\P_tv=\Lap v+V(x)v^p,\quad v|_{t=0}=u(\t)>0.
\end{cases}
\end{align*}
By the comparison principle, it follows that $w\leq v$ hence
\begin{align}
u_{\ast V}\left([t-2\t]_+\right)\leq u\left(t+\t\right)\quad\mbox{for all $\t>0$}.
\end{align}
Taking $\t\to0$, we conclude that $u\geq u_{\ast V}$, hence $u=u_{\ast V}$.\smallskip

Assume now that $\kappa>0$. According to Lemma \ref{Lem:Semizero}, then $u^\kappa$ solves
\begin{align*}
\P_tu^\kappa-\Lap\P_tu^\kappa=\Lap u^\kappa+V(x)(u^\kappa)^p,\quad u^\kappa|_{t=0}=0.
\end{align*}
We can conclude by the previous case that $u^\kappa=u_{\ast V}$. This implies
\begin{align}
u(x,t)=u_{\ast V}(x,t-\kappa)\quad\mbox{for all $x\in\R^n,t>\kappa$}.
\end{align}
Since $u(t)\equiv0$ for $0\leq t<\kappa$, we can conclude that $u(x,t)=u_{\ast V}\left(x,[t-\kappa]_+\right)$.\QED
\end{proof}

\begin{corollary}
Assume $V$ satisfies (\ref{Hyp:New}) with $\nu=0,\S\in J_n^+$, and $0<p<\V_0^{-1}$. A function $0\leq u\in\CZ_{\I,a}$ with $a\geq\frac{\S}{1-p}$ is a nontrivial s-solution of (\ref{Eqn:zero}) if and only if 
\[
u=u_{\ast V}\left([t-\kappa]_+\right)\quad\mbox{for some $\kappa\geq0$}.
\]
\end{corollary}

Now we explore the more general case that $V$ is both space and time-dependent. 

\begin{theorem}[Classification III]\label{Thm:Classf3}

Assume (\ref{HypNonUnq}) 
and, in addition, that $V$ satisfies the following uniform convexity condition in $t$: $\exists$ a continuous function $\A(\t)>0$ such that $\A(\t)\to1$ as $\t\to0$ and
\begin{align}\label{Hyp:Convex}
V(x,t+\t)\geq\A(\t)V(x,t)\quad(x\in\R^n,t,\t>0).
\end{align}
Then $0\leq u\in\CZ_{\I,a}$ is a nontrivial s-solution of (\ref{Eqn:zero}) if and only if $u=u_{\ast V^{\kappa}}\left(x,[t-\kappa]_+\right)$ for some $\kappa\geq0$. 
\end{theorem}

\begin{proof}
%
We proceed as in the preceding theorem. First, we note that the converse part, i.e.\ $u_{\ast V^{\kappa}}\left([\cdot-\kappa]_+\right)$ is an s-solution of (\ref{Eqn:zero}) with $V$ replaced by $V^{\kappa}$, is true by Lemma \ref{Lem:Semizero}. \smallskip

Now we prove the direct part. Assume $u$ is a nontrivial s-solution of (\ref{Eqn:zero}). We set $\kappa$ by (\ref{NonUniquet0}). Let us consider the case $\kappa=0$. By the same reasoning as before, then we have $u(x,t)>0$ for all $x\in\R^n$ and $t>0$. Let $\t>0$. Define 
\begin{align}
K=(1\wedge\A(\t))^{-q}\geq1\quad\Rightarrow\quad V\leq K^{1-p}V^\t.
\end{align}
Let $\T u\triangleq u^\t$, $\T\mu\triangleq e^t\T u$. We calculate
\begin{align}
\CM^s_{(e^{-\t}\mu(\t),V)}(K\T\mu)&=e^{-\t}\mu(\t)+\int_0^t\CB(K\T\mu)(s)\,ds+\int_0^te^{(1-p)s}\CB\left(VK^p\T\mu^p\right)(s)ds,\nonumber\\
&\leq K\left(e^{-\t}\mu(\t)+\int_0^\t\CB\T\mu(s)ds+\int_0^te^{(1-p)s}\CB\left(V^\t\T\mu^p\right)(s)ds\right)=K\T\mu.
\end{align}
Thus $\A(\t)^{-q}u^\t$ is an s-supersolution to Eq.\ (\ref{Eqn:Main}) with initial value $e^{-\t}u(\t)>0$.\smallskip

Let $k>0$ be a constant to be specified. Let $v\triangleq u_{\ast V}([t-\t]_+)$ and $\psi=e^tv$, then $\psi\equiv0$ if $t\leq\t$, and for $t>\t$, we have
\begin{align}
\CM^s_{(0,V)}(k\psi)&=\int_\t^t\CB(k\psi)(s)ds+\int_\t^te^{(1-p)s}\CB\left(Vk^p\psi^p\right)(s)ds,\nonumber\\
&=k\int_0^{t-\t}\CB\left(e^{s+\t}u_{\ast V}\right)ds+k^p\int_0^{t-\t}e^{(1-p)(s+\t)}\CB\left(V^\t e^{(s+\t)p}u_{\ast V}^p\right)ds,\nonumber\\
&=ke^\t\int_0^{t-\t}\CB(e^su_{\ast V})ds+k^pe^\t\int_0^{t-\t}e^{(1-p)s}\CB\left(V^\t e^{sp}u_{\ast V}^p\right)ds,\nonumber\\
&\geq ke^\t\int_0^{t-\t}\CB\T u\,ds+k^pe^\t\A(\t)\int_0^{t-\t}e^{(1-p)s}\CB\left(V\T u^p\right)ds,\nonumber
\end{align}
where
\[
\T u(x,t)\triangleq e^tu_{\ast V}(x,t).
\]
By taking $k=\A(\t)^q\leq1$, we find that $k^p\A(\t)=k$ and hence
\begin{align*}
\CM^s_{(0,V)}(k\psi)&\geq ke^\t\left(\int_0^{t-\t}\CB\T u\,ds+\int_0^{t-\t}e^{(1-p)s}\CB\left(V\T u^p\right)ds\right),\\
&=ke^\t\T u(t-\t)=k\psi.
\end{align*}
Thus $\A(\t)^qv$ is an s-subsolution to (\ref{Eqn:Main}) with zero initial value. By the comparison theorem (Theorem \ref{Thm:Comp}), we get that
\begin{align}
\A(\t)^qu_{\ast V}(x,t-\t)\leq\A(\t)^{-q}u(x,t+\t)\quad\forall\,x\in\R^n,t\geq\t.
\end{align}
Taking $\t\to0$ we obtain 
\[
u_{\ast V}(x,t)\leq u(x,t)\quad\forall\,x\in\R^n,t>0.
\]
Clearly $u\leq u_{\ast V}$, therefore $u=u_{\ast V}$ which proves the desired result when $\kappa=0$.\smallskip

Assume now that $\kappa>0$. By Lemma \ref{Lem:Semizero}, $u^\kappa(0)=u_{\ast V^{\kappa}}(0)=0$, $u^\kappa(t),u_{\ast V^\kappa}(t)>0$ for $t>0$, and $u^{\kappa}$, $u_{\ast V^{\kappa}}$ satisfy (\ref{Eqn:zero}) with the same potential function $V^\kappa$:
\begin{align}
w=\CM^s_{(0,V^{\kappa})}w.
\end{align}
By the previous case $\kappa=0$, we have $u^\kappa=u_{\ast V^\kappa}$ hence
\begin{align}
u(x,t+\kappa)=u_{\ast V^{\kappa}}(x,t)\quad\mbox{for all $x\in\R^n,t\geq0$}.
\end{align}
Since $u(x,t)=0$ whenever $0\leq t<\kappa$, we conclude that $u(x,t)=u_{V^{\kappa}}(x,[t-\kappa]_+)$.\QED
\end{proof}

\begin{corollary}

Let $0<p<1$, $\S\in J_n^+$, and $\L(t)>0$ satisfies $\L(t+\t)\geq\A(\t)\L(t)$ where $\A\in C(\R;\R_{>0})$, $\A(\t)\to1$ as $\t\to0$. A function $0\leq u\in\CZ_{\I,a}$ is a nontrivial solution of the problem
\begin{align*}
\begin{cases}
\P_tu-\Lap\P_tu=\Lap u+\L(t)|x|^\S u^p&x\in\R^n,t>0,\\
\qquad u(x,0)=0&x\in\R^n,
\end{cases}
\end{align*}
if and only if there is a constant $\kappa\geq0$ such that $u^\kappa$ is the unique maximal solution of the same problem but the potential is replaced by $\L(t+\kappa)|x|^\S$.

\end{corollary}

\begin{remark}

If $V(x,t)$ is non-decreasing in $t$ for each $x$, then clearly we can take $\A\equiv1$ in (\ref{Hyp:Convex}) so that the convexity condition is true. Thus the following potentials
\[
Ct^k|x|^\S,\quad Ct^k(\log t)^\nu|x|^\S\quad(k,\nu\geq0,\S\in J_n^+)
\]
and generally,
\[
\L(t)|x|^\S\quad\mbox{with $\L'(t)\geq0$}
\]
satisfy the convexity condition.

\end{remark}

\section{Asymptotic grow-up rates}\label{Sec:Grow-up}

In Theorem \ref{Thm:SharpLower}, we have obtained the lower asymptotic grow-up rate of solutions to (\ref{Eqn:Main}) when the initial condition $u_0$ is only assumed to be non-trivial. In this section, by imposing the precise behavior at infinity of $u_0$, we will get both the sharp upper and sharp lower asymptotic behaviors of the solutions. More precisely, we assume (\ref{HypSharp}). This implies 
$u_0\in BC_a$ and $u_0\neq0$.\smallskip

Regarding the potential $V(x,t)$, it is assumed to have the form 
\begin{align}\label{Hyp:Vgrowup}
\begin{cases}
\displaystyle V=\Z(t)(1+t)^\nu|x|^\S\quad(\mbox{$x\in\R^n$, $t>0$}),\\
\vspace{-10pt}\\
\displaystyle\mbox{where}\,\,\lim_{t\to\I}t^\D\Z(t)=\I,\quad\lim_{t\to\I}t^{-\D}\Z(t)=0\quad(\forall\,\D>0)
\end{cases}
\end{align}
here $\nu\in\R,\S\in J_n^+$ and $\Z(t)(1+t)^\nu\in\CE_\V$ in the case $\S>0$ (see (\ref{Hyp:epsilonB})). 
Let
\begin{align}
a_0=\frac{\S}{1-p},\quad a_c=\frac{\S+2(\nu+1)_+}{1-p}.
\end{align}
If we set $\T a=\max\{a_0,a\}$ then it follows from the global existence (Theorem \ref{Thm:Existence}) and uniqueness (Theorem \ref{Thm:Unique}) results that Eq.\ (\ref{Eqn:Main}) admits a unique global s-solution $u\in\CZ_{\I,\T a}$.\smallskip

Our first main result of this section is the asymptotic behavior of solutions when the spatial growth exponent $a$ of the initial condition is above the critical value $a_c$.

\begin{theorem}[Supercritical, $a>a_c$]\label{Thm:a>ac}

Assume (\ref{HypSharp}), $n\geq1$, and 
\begin{align}
a\in\left(a_c,\I\right).
\end{align}
Let $u$ be the global s-solution of (\ref{Eqn:Main}) where $V$ has the form (\ref{Hyp:Vgrowup}) and $0<p<1$. Then
\begin{align}
\U Cl_1t^{\frac{a}{2}}\leq \|u(\cdot,t)\|_{R,a}\leq \O Cl_2t^{\frac{a}{2}}\quad\mbox{as $t\to\I$},
\end{align}
for some constants $\U C,\O C>0$.

\end{theorem}

\begin{proof}

We prove the lower bound. By the first part of (\ref{HypSharp}), there is $r_1>0$ such that $u_0(x)\geq\frac{l_1}{2}\left(1+|x|^2\right)^{a/2}$ for all $|x|\geq r_1$. This implies
\begin{align*}
u_0(x)\geq\frac{l_1}{2}\left(\left(1+|x|^2\right)^{\frac{a}{2}}-\left(1+r_0^2\right)^{\frac{a}{2}}\right)\quad\mbox{for all}\,\,x\in\R^n.
\end{align*}
Using Proposition \ref{Prop:TwoSB} (2) and noting that $a>2$, there is a constant $c_1=c_1(n,a)>0$ such that
\begin{align}
\CG(t)u_0&\geq\frac{l_1}{2}\left(c_1\left(1+t+|x|^2\right)^{\frac{a}{2}}-\left(1+r_0^2\right)^{\frac{a}{2}}\right),\nonumber\\
&\geq\frac{c_1l_1}{4}\left(1+t+|x|^2\right)^{\frac{a}{2}},\label{Tmp:Gtu0Lower}
\end{align}
provided $t\geq t_1\triangleq(4/c_1)^{2/a}\left(1+r_0^2\right)-1$. 
Since $u\geq\CG(t)u_0$, we get that
\begin{align}
\|u(\cdot,t)\|&=\sup_x\left(R^2+|x|^2\right)^{-\frac{a}{2}}|u(x,t)|\nonumber\\
&\geq\frac{c_1l_1}{4}\sup_x\left(\frac{1+t+|x|^2}{R^2+|x|^2}\right)^{\frac{a}{2}},\label{SupCrit1}\\
&\geq R^{-a}\frac{c_1l_1}{4}(1+t)^{\frac{a}{2}}\geq\U Cl_1t^{\frac{a}{2}}\quad\mbox{as $t\to\I$},\nonumber
\end{align}
where $\U C=R^{-a}c_1/4$. We have used that, for $t\gg1$, $(1+t+|x|^2)/(R^2+|x|^2)
$ is decreasing in $|x|$.
\smallskip

Next we prove the upper bound by constructing a super-solution. Employing the second part of (\ref{HypSharp}), there is $r_2>0$ such that $u_0(x)\leq2l_2\left(1+|x|^2\right)^{a/2}$ for $|x|\geq r_2$. Set $c_2=2l_2(1+r_2^2)^{a/2}\triangleq c_2'l_2$. Then
\begin{align}\label{AbigAc0}
u_0(x)\leq c_2\left(1+|x|^2\right)^{\frac{a}{2}}\quad\mbox{for all $x\in\R^n$}.
\end{align}
Let $\r\geq1$ to be specified. We introduce the function
\begin{align}
U=c_2\left(\r+\r t+|x|^2\right)^{\frac{a}{2}}.
\end{align}
If $\r=\r(n,a)$ is sufficiently large 
then we have as in the proof of Proposition \ref{Prop:TwoSB} (1) (with $\r=k$) that
\begin{align*}
\P_tU-\P_t\Lap U\geq\Lap U.
\end{align*}

Let $\psi=\psi(t)>0$ be a differentiable function to be specified such that $\psi'\geq0$ and define
\begin{align}
w=\psi U.
\end{align}
We calculate
\begin{align*}
\P_tw-\P_t\Lap w&=\psi\left(\P_tU-\P_t\Lap U\right)+\psi'\left(U-\Lap U\right),\nonumber\\
&\geq \Lap w+\psi'\left(U-\Lap U\right).
\end{align*}
Our goal is to select $\psi$ in such a way that 
\[
\psi'(U-\Lap U)\geq V\psi^pU^p.
\]

Let $a=a_0+2\gamma$ where $\gamma>0$ and denote $\VS=\r+\r t+r^2$ so $U=c_2\VS^{a/2}$. We note that 
\begin{align*}
\frac{a}{2}(1-p)=\frac{\S}{2}+\gamma(1-p).
\end{align*}
For $r>0$, using that $V=\Z(t)(1+t)^\nu r^\S$ one can calculate to get
\begin{align*}
\frac{U-\Lap U}{VU^p}&=(\Z(t)(1+t)^{\nu})^{-1}c_2^{1-p}\left(\frac{\VS^{\frac{\S}{2}}}{r^\S}\right)\VS^{\gamma(1-p)}\left(1-\left(2n+(2a-4)r^2\VS^{-1}\right)\frac{a}{2}\VS^{-1}\right),
\end{align*}
Note that 
\begin{align*}
r^2\VS^{-1}\in[0,1]\quad\mbox{and}\quad \VS^{-1}\in[0,1/\r].
\end{align*}
Choose $\r=\r(n,a)>0$ sufficiently large, then we have
\begin{align*}
\begin{cases}
\displaystyle 1-(2n+(2a-4)r^2\VS^{-1})\frac{a}{2}\VS^{-1}\geq\frac{1}{2},\\
\vspace{-10pt}\\
\displaystyle\frac{\VS^{\frac{\S}{2}}}{r^\S}=\left(1+\frac{\r+\r t}{r^2}\right)^{\frac{\S}{2}}\geq1,\\
\vspace{-10pt}\\
\displaystyle(\Z(t)(1+t)^{\nu})^{-1}\VS^{\gamma(1-p)}\geq \r^{\gamma(1-p)}\Z(t)^{-1}(1+t)^{e},
\end{cases}
\end{align*}
where
\begin{align}
e\triangleq \gamma(1-p)-(\nu+1)_++1.
\end{align}
Hence
\begin{align}\label{AbigAc1}
\frac{U-\Lap U}{VU^p}\geq c_3\Z(t)^{-1}(1+t)^{e}\quad\mbox{where}\,\,c_3\triangleq\frac{c_2^{1-p}}{2}\r^{\gamma(1-p)}.
\end{align}

We now choose $\psi>0$ by requiring that 
\begin{align}\label{AbigAc2}
\psi(0)=1,\quad\psi'=c_3^{-1}\Z(t)(1+t)^{-e}\psi^p.
\end{align}
Since $a>a_c$, we have $2\gamma=a-a_0>a_c-a_0=2(\frac{(\nu+1)_+}{1-p})$ so $e=\gamma(1-p)-(\nu+1)_++1>1$. 
Also, we note that $\Z(t)\leq C_\D(1+t)^\D$ for all $\D>0$, so 
\[
\Z(t)(1+t)^{-e}\leq C_\D(1+t)^{-e+\D}\quad\mbox{with}\,\,-e+\D<-1,
\]
for $\D>0$ sufficiently small. Fix such a $\D>0$. By solving the initial value problem, we obtain
\begin{align*}
\psi(t)=\left(1+\frac{1-p}{c_3}\int_0^t\Z(\t)(1+\t)^{-e}d\t\right)^{\frac{1}{1-p}}.
\end{align*}
The function $\psi$ has the upper bound
\begin{align*}
\psi(t)\leq\L_0\triangleq\left(1+\frac{1-p}{c_3}\frac{C_\D}{\gamma(1-p)-(\nu+1)_+-\D}\right)^{\frac{1}{1-p}}.
\end{align*}

By (\ref{AbigAc1}) and (\ref{AbigAc2}), we have $\psi'(U-\Lap U)\geq Vw^p$, hence now we obtain
\begin{align}
\P_tw-\P_t\Lap w\geq\Lap w+Vw^p\quad\mbox{in $Q_\I$},\quad w|_{t=0}\geq u_0.
\end{align}
Applying the comparison theorem (Corollary \ref{Cor:Compare}), we conclude that
\begin{align}
u\leq w\leq\L_0 c_2\left(\r+\r t+|x|^2\right)^{\frac{a}{2}}.\label{SupCrit2}
\end{align}
We take the norm
\begin{align*}
\|u(t)\|&\leq\L_0 c_2\sup_x(R+|x|^2)^{-\frac{a}{2}}(\r+\r t+|x|^2)^{\frac{a}{2}},\nonumber\\
&=\L_0 c_2\sup_x\left(\frac{\r+\r t+|x|^2}{R^2+|x|^2}\right)^{\frac{a}{2}}\leq\U Cl_2(1+t)^{\frac{a}{2}}\quad\mbox{as $t\to\I$},
\end{align*}
where $\O C=R^{-a}\L_0 c_2'\r^{a/2}$.\QED

\end{proof}




In the case $a=a_c$, the asymptotic property of solutions is described by the following theorem.


\begin{theorem}[Critical, $a=a_c$]\label{Thm:a=a_c}

Assume (\ref{HypSharp}), $n\geq1$, and 
\begin{align}
a=a_c=\frac{\S+2(\nu+1)_+}{1-p}.
\end{align}
Let $u$ be the global s-solution of (\ref{Eqn:Main}) where $V$ has the form (\ref{Hyp:Vgrowup}) and $0<p<1$. Then
\begin{align}
\U Cl_1t^{\frac{a}{2}}\leq \|u(\cdot,t)\|_{R,a}\leq \O C_\V l_2t^{\frac{a}{2}+\V}\quad\mbox{as $t\to\I$},
\end{align}
for some constants $\U C,\O C_\V>0$.

\end{theorem}

\begin{proof}

By following the proof in the preceding theorem, we get the lower bound
\begin{align*}
u(x,t)\geq \frac{c_1l_1}{4}(1+t+|x|^2)^{\frac{a_c}{2}}\quad(\mbox{as $t\to\I$}).
\end{align*}
For the upper bound, we can employ the same calculations. In fact, we have $e=1$ in this case and the function $\psi$, for a fixed $\D>0$ small, is bounded by
\[
\psi(t)\leq\left(1+\frac{1-p}{c_3}C_\D t^\D\right)^{\frac{1}{1-p}}.
\]
So for any $\V=\D/(1-p)>0$ sufficiently small we obtain
\begin{align*}
u(x,t)&\leq c_2''l_2\left(1+\frac{1-p}{c_3}C_\D t^\D\right)^{\frac{1}{1-p}}(1+t+|x|^2)^{\frac{a_c}{2}},\\
&\leq C_\V l_2(1+t+|x|^2)^{\frac{a_c}{2}+\V},
\end{align*}
as $t\to\I$. We therefore obtain
\begin{align}
\U Cl_1t^{\frac{a_c}{2}}\leq\|u(\cdot,t)\|_{R,a}\leq\O C_\V l_2t^{\frac{a_c}{2}+\V}
\end{align}
for any $\V>0$ where $C_\V$ is a constant depending on $\V$.\QED

\end{proof}

Finally we consider the sub-critical case: $a<a_c$. If the asymptotic growth $a$ of $u_0$ (see (\ref{HypSharp})) satisfies $a<a_0$, the result will be the same as replacing $a$ with $a_0$ in the following theorem. It is remarkable that for the sub-critical exponent case, the initial condition plays no role in the asymptotic estimates. In other words, the asymptotic behavior of solutions in this case is dominated by the sublinearity and the potential not by the initial condition.

\begin{theorem}[Sub-critical, $a<a_c$]\label{Thm:a<ac}

Assume (\ref{HypSharp}), $n\geq1$, and 
\[
a\in[a_0,a_c).
\]
Let $u$ be the global s-solution of (\ref{Eqn:Main}) where $V$ has the form (\ref{Hyp:Vgrowup}) and $0<p<1$. Then, for any $\V>0$, there are constants $\U C,\O C_\V>0$ such that
\begin{align}
\U Ct^{\frac{a_c}{2}}\leq\|u(\cdot,t)\|_{R,a}\leq\O C_\V l_2^\V t^{\frac{a_c}{2}+\V}\quad\mbox{as $t\to\I$}.
\end{align}

\end{theorem}

\begin{proof}
The lower bound was proved in Theorem \ref{Thm:SharpLower}. We prove the upper bound. It is remarkable that the technique from the preceding theorem cannot be applied. For simplicity of the presentation, we will consider the case $\nu>-1$. \smallskip

Arguing the same as the preceding theorem, we can assume that (\ref{AbigAc0}) is true. Let $R\geq k>1$ be large constants such that Proposition \ref{Prop:TwoSB} (1) and Lemma \ref{Lem:BesSup} (for a fixed $\theta$) hold and the following estimates are true
\begin{align}
&\CG(t)\left(R+\r+|x|^2\right)^{\frac{m}{2}}\leq\left(R+\r+kt+|x|^2\right)^{\frac{m}{2}},\label{A0Ac:1}\\
&\CB\left(R+\r+|x|^2\right)^{\frac{m}{2}}\leq\theta^{-1}\left(R+\r+|x|^2\right)^{\frac{m}{2}},\label{A0Ac:2}
\end{align}
for all $(x,t)\in\R^n\times(0,\I)$, $\r\geq0$, and $m$ in a compact subset of $[0,\I)$. In the following $BC_a$ will be equipped with the norm $\|\cdot\|\triangleq\|\cdot\|_{\sqrt{R},a}$.\smallskip

We shall choose $b>0$ sufficiently small later and then note that
\[
\Z(t)\leq\Vt(1+t)^b\quad\mbox{where}\,\,\Vt=\Vt(b).
\]
Let $\T\nu=\nu+b$ so that we have
\[
V(x,t)\leq\Vt(1+t)^{\T\nu}|x|^\S.
\]
From the definition of solutions we have
\begin{align}
u&=\CG(t)u_0+\int_0^t\CG(t-\t)\CB\left(Vu^p\right)(\t)\,d\t\nonumber\\
&\leq c_2
\CG(t)\left(1+|x|^2\right)^{\frac{a}{2}}+\int_0^t\CG(t-\t)\CB\left(\vartheta(1+\t)^{\T\nu}|x|^\S\left(\|u(\cdot,\t)\|\left(R+|x|^2\right)^{a/2}\right)^p\right)d\t,
\nonumber\\
&\leq c_2\CG(t)\left(R+|x|^2\right)^{\frac{a}{2}}+\vartheta\sup_{[0,t]}\|u(\cdot,\t)\|^p\int_0^t(1+\t)^{\T\nu}\CG(t-\t)\CB\left(R+|x|^2\right)^{\frac{\S+ap}{2}}d\t,\nonumber\\
&\leq c_2\left(R+kt+|x|^2\right)^{\frac{a}{2}}+\vartheta\theta^{-1}\sup_{[0,t]}\|u(\cdot,\t)\|^p\int_0^t(1+\t)^{\T\nu}\left(R+k(t-\t)+|x|^2\right)^{\frac{\S+ap}{2}}d\t,\nonumber\\
&\leq c_2\left(R+kt+|x|^2\right)^{\frac{a}{2}}+\vartheta\theta^{-1}\left(\sup_{[0,t]}\|u(\cdot,\t)\|\right)^p\left(R+kt+|x|^2\right)^{\frac{\S+ap}{2}}(1+t)^{\T\nu+1},\label{A0Ac:3}\\
&\leq c_2\left(R+|x|^2\right)^{\frac{a}{2}}(1+t)^{\frac{a}{2}}+\Vt\th^{-1}\left(\sup_{[0,t]}\|u(\cdot,\t)\|\right)^p\left(R+|x|^2\right)^{\frac{a}{2}}\left(1+t\right)^{\frac{\S+ap}{2}+\T\nu+1},\label{A0Ac:4}
\end{align}
where we have used (\ref{A0Ac:1}), (\ref{A0Ac:2}) (with $\r=0$) in the third estimate and that $\S+ap\leq a$, which is true because $a\geq a_0$, in the last estimate. \smallskip

Let 
\[
\D\triangleq\frac{\S+ap}{2}+\T\nu+1,\quad M\triangleq\max\{1,c_2+\Vt\th^{-1}\},\quad K\triangleq\Vt\th^{-1}.
\]
Since $a<\frac{\S+2(\nu+1)}{1-p}$, it follows that
\begin{align}\label{Tmp:aDelta}
\frac{a}{2}<\D.
\end{align}

\begin{claim}
We have
\begin{align}
\|u(\cdot,t)\|\leq M^q(1+t)^{\D q}\quad(\forall\,t>0).\label{A0Ac:5}
\end{align}

\end{claim}

\begin{proof}[\textbf{Claim}]

If $\|u(\cdot,t)\|\leq1$ then the claim is trivially true. Assume $\|u(\cdot,t)\|\geq1$. Then 
\[
\sup_{[0,t_1]}\|u(\cdot,\t)\|\geq1\quad\mbox{for $t_1>t$}.
\]
By (\ref{A0Ac:4}) and that $a/2<\D$, we get
\[
\|u(\cdot,t)\|\leq M\left(\sup_{[0,t_1]}\|u(\cdot,\t)\|\right)^p(1+t)^\D\,\,\,\,(\forall\,t<t_1)\quad\Rightarrow\quad\sup_{[0,t_1]}\|u(\cdot,\t)\|\leq M^q(1+t_1)^{\D q}.
\]
This is true for any $t_1>t$, hence
\begin{align*}
\|u(\cdot,t)\|\leq M^q(1+t)^{\D q}.\QED
\end{align*}
\end{proof}

Without loss of generality we assume $c_2\geq1$. We introduce the sequence of real numbers:
\begin{align}
\Z_0=a,\quad\Z_1=\S+ap,\quad\mbox{and}\quad\Z_k=\S+\Z_{k-1}p\quad(k\geq1).
\end{align}
Also for convenience, let 
\begin{align}
\VS\triangleq R+kt+|x|^2,\quad \VP\triangleq 1+t.
\end{align}
By (\ref{A0Ac:3}), (\ref{Tmp:aDelta}), and (\ref{A0Ac:5}), we have
\begin{align}
u(x,t)&\leq c_2\VS^{\frac{\Z_0}{2}}+K\left(\sup_{0\leq\t\leq t}\|u(\cdot,\t)\|\right)^p\VS^{\frac{\Z_1}{2}}\VP^{\T\nu+1},\nonumber\\
&\leq c_2\VS^{\frac{\Z_0}{2}}+K\left(M^q\VP^{\D q}\right)^p\VS^{\frac{\Z_1}{2}}\VP^{\T\nu+1},\nonumber\\
&\leq \CI_1\triangleq d_1\VS^{\frac{\Z_0}{2}}+m_1\VP^{\D qp+(\T\nu+1)}\VS^{\frac{\Z_1}{2}}\quad\mbox{where}\,\, d_1=c_2,\,\, m_1\triangleq KM^{qp}.\label{A0Ac:6}
\end{align}
We shall use the following fact. For $0\leq\t\leq t$, we have 
\begin{align}
&\CG(t-\t)\CB\left(V\VS^{\frac{m}{2}}\right)(\t)\leq\vartheta(1+\t)^{\T\nu}\CG(t-\t)\CB\left(\VS(\t)^{\frac{\S+m}{2}}\right),\nonumber\\
&\hphantom{\CG(t-\t)\CB\left(V\VS(\t)^{\frac{m}{2}}\right)}\leq\vartheta\theta^{-1}\VS^{\frac{\S+m}{2}}\VP^{\T\nu}
=K\VS^{\frac{\S+m}{2}}\VP^{\T\nu},
\label{A0Ac:7}
\end{align}
which is true according to (\ref{A0Ac:1}) and (\ref{A0Ac:2}) (taking $\r=k\t$). Note that $\VS=\VS(t)$.\smallskip

Then by plugging the estimate $u\leq \CI_1$ into the integral representation of $u$ we get
\begin{align}
u(x,t)&\leq c_2\VS^{\frac{\Z_0}{2}}+\int_0^t\CG(t-\t)\CB\left(V\CI_1^p\right)(\t)\,d\t,\nonumber\\
&\leq c_2\VS^{\frac{\Z_0}{2}}+2^p\int_0^t\CG(t-\t)\CB\left(d_1^p(V\VS^{\frac{\Z_0p}{2}})+m_1^p\VP^{\D qp^2+(\T\nu+1)p}(V\VS^{\frac{\Z_1p}{2}})\right)(\t)\,d\t,\nonumber\\
&\leq c_2\VS^{\frac{\Z_0}{2}}+2^pK\int_0^t\left(d_1^p\VS^{\frac{\Z_1}{2}}\VP^{\T\nu}+m_1^p\VP^{\D qp^2+(\T\nu+1)p}\VS^{\frac{\Z_2}{2}}\VP^{\nu}\right)d\t,\nonumber\\
&\leq d_2\left(\VS^{\frac{\Z_0}{2}}+\VS^{\frac{\Z_1}{2}}\VP^{\T\nu+1}\right)+m_2\VP^{\D qp^2+(\T\nu+1)(1+p)}\VS^{\frac{\Z_2}{2}}\triangleq \CI_2,
\end{align}
where
\begin{align*}
d_2=2^pKc_2,\quad m_2\triangleq 2^pK^{1+p}M^{qp^2}.
\end{align*}
Here we use the fact that $(a_1+\cdots+a_m)^p\leq m^p(a_1^p+\cdots+a_m^p)$ for $a_i\geq0$. Similarly, we have
\begin{align}
u&\leq c_2\VS^{\frac{\Z_0}{2}}+\int_0^t\CG(t-\t)\CB\left(V\CI_2^p\right)(\t)\,d\t,\nonumber\\
&\leq c_2\VS^{\frac{\Z_0}{2}}+3^p\int_0^t\CG(t-\t)\CB\Big(d_2^p(V\VS^{\frac{\Z_0p}{2}})+d_2^p(V\VS^{\frac{\Z_1p}{2}})\VP^{(\T\nu+1)p}\nonumber\\
&\hspace{6cm}+m_2^p\VP^{\D qp^3+(\T\nu+1)(p+p^2)}(V\VS^{\frac{\Z_2p}{2}})\Big)(\t)\,d\t,\nonumber\\
&\leq c_2\VS^{\frac{\Z_0}{2}}+3^pK\int_0^t\left(d_2^p\VS^{\frac{\Z_1}{2}}\VP^{\T\nu}+d_2^{p}\VS^{\frac{\Z_2}{2}}\VP^{\T\nu+(\T\nu+1)p}+m_2^p\VP^{\D qp^3+(\T\nu+1)(p+p^2)+\T\nu}\VS^{\frac{\Z_3}{2}}\right)d\t,\nonumber\\
&\leq d_3\left(\VS^{\frac{\Z_0}{2}}+\VS^{\frac{\Z_1}{2}}\VP^{\T\nu+1}+\VS^{\frac{\Z_2}{2}}\VP^{(\T\nu+1)(1+p)}\right)+m_3\VP^{\D qp^3+(\T\nu+1)(1+p+p^2)}\VS^{\frac{\Z_3}{2}}\triangleq \CI_3,
\end{align}
where
\begin{align*}
d_3=3^p2^{p^2}K^{1+p}c_2,\quad m_3=
3^p2^{p^2}K^{1+p+p^2}M^{qp^3}.
\end{align*}
By iterating $N$ times we obtain
\begin{align}
u(x,t)\leq I_{N}\triangleq d_N\sum_{k=0}^{N-1}\VS^{\frac{\Z_k}{2}}\VP^{(\T\nu+1)(1+p+\cdots+p^{k-1})}+m_N\VP^{\D qp^N+(\T\nu+1)(1+p+\cdots+p^{N-1})}\VS^{\frac{\Z_N}{2}},\label{A0Ac:8}
\end{align}
where 
\begin{align*}
d_N=\left(\prod_{k=2}^Nk^{p^{N+1-k}}\right)K^{(1+p+\cdots+p^{N-2})}c_2,\quad m_N\triangleq \left(\prod_{k=2}^Nk^{p^{N+1-k}}\right)K^{1+p+\cdots+p^{N-1}}M^{qp^N}.
\end{align*}

We observe that 
\begin{align*}
\Z_k=\S+\S p+\S p^2+\cdots+\S p^{k-1}+ap^k=\frac{\S}{1-p}+\left(a-\frac{\S}{1-p}\right)p^k.
\end{align*}
Since $\frac{\S}{1-p}\leq a<\frac{\S+2(\nu+1)}{1-p}$ and $0<p<1$, it follows that 
\begin{align}\label{A0Ac:9}
\begin{cases}
\displaystyle
\frac{\S}{1-p}\leq\Z_k\leq a,\\
\vspace{-10pt}\\
\displaystyle\frac{\Z_k}{2}+(\T\nu+1)\left(1+p+\cdots+p^{k-1}\right)
\leq\frac{a_c}{2}
\end{cases}
\end{align}
by taking $b>0$ sufficiently small. Using (\ref{A0Ac:9}) and (\ref{A0Ac:8}) together with that 
\[
\VS\leq(R+|x|^2)(1+t)
\]
we obtain
\begin{align*}
u(x,t)&\leq (R+|x|^2)^{\frac{a}{2}}\left(Nd_{N}(1+t)^{\frac{a_c}{2}}+m_N(1+t)^{\frac{a_c}{2}+\D qp^N}\right).
\end{align*}
Choose $N$ sufficiently large so that
\begin{align*}
N\geq\frac{\log(\V/((1\vee\D)q))}{\log p}\quad\Rightarrow\quad(1\vee\D)qp^N\leq\V.
\end{align*}
Then by taking $t$ sufficiently large, it follows that
\begin{align*}
u(x,t)\leq 2m_N(R+|x|^2)^{\frac{a}{2}}t^{\frac{a_c}{2}+\V}\leq C_\V l_2^\V(R+|x|^2)^{\frac{a}{2}}t^{\frac{a_c}{2}+\V}
\end{align*}
where $C_\V=2c_{N,p}K^q(c_2')^{qp^N}$. Therefore, we obtain
\begin{align}
\|u(\cdot,t)\|\leq\O C_\V l_2^\V t^{\frac{a_c}{2}+\V}
\end{align}
as $t\to\I$, where $\O C_\V=C_\V R^{-a}$. This proves the desired upper estimate.\QED
\end{proof}

\begin{remark}

We have discovered that if the spatial growth exponent $a$ of $u_0$ is in the interval $(-\I,a_c)$, then  the sublinear source takes control of the asymptotic grow-up rate of the solutions with the grow-up rate $\sim t^{a_c/2+\V}$ for any $\V>0$. At the critical exponent $a=a_c$, both the sublinear source and the initial value control the grow-up rate of the solutions. The initial condition takes over the control of the grow-up rate when $a$ is super-critical: $a\in(a_c,\I)$. 
\end{remark}

\begin{remark}\label{Rem:VisEf}

It is interesting to investigate the  effect of the third order viscous term $\Lap\P_tu$ in (\ref{Eqn:Main}) on the asymptotic properties of solutions. We consider non-trivial solutions $u\geq0$ of the Cauchy problem
\begin{align}\label{Eqn:Sobk}
\begin{cases}
\P_tu-k\Lap\P_tu=\Lap u+Vu^p\quad(\mbox{in $Q_\I$}),\\
\qquad u(x,0)=u_0(x),
\end{cases}
\end{align}
where $k>0$ is a constant and $u_0>0$ has $l_1,l_2\in(0,\I)$ given by
\begin{align}
l_1=\liminf_{|x|\to\I}\frac{u_0(x)}{|x|^a},\quad l_2=\limsup_{|x|\to\I}\frac{u_0(x)}{|x|^a}.
\end{align}
Setting
\begin{align}
U(x,t)=k^{-q}u\left(\sqrt{k}x,kt\right)\quad\left(q=\frac{1}{1-p}\right),
\end{align}
then $\P_tU=k^{1-q}\P_tu$, $\Lap\P_tU=k^{1-q}(k\Lap\P_tu)$, $\Lap U=k^{1-q}\Lap u$, and $VU^p=k^{1-q}(Vu^p)$, hence $U$ satisfies
\begin{align}\label{Eqn:Sobnk}
\begin{cases}
\P_tU-\Lap\P_tU=\Lap U+VU^p&(\mbox{in $Q_\I$}),\\
\qquad U(x,0)=U_0(x)\triangleq k^{-q}u_0.
\end{cases}
\end{align}
Conversely, given any solution $U$ of Eq.\ (\ref{Eqn:Sobnk}), then
\begin{align*}
u(x,t)=k^qU\left(\frac{x}{\sqrt{k}},\frac{t}{k}\right)
\end{align*}
solves Eq.\ (\ref{Eqn:Sobk}). 
\smallskip

By setting $y=x/\sqrt{k}$ and $\t=t/k$, we have
\begin{align}
\|u(\cdot,t)\|_{R,a}&=\sup_x\left(R^2+|x|^2\right)^{-a/2}u(x,t),\nonumber\\
&=k^q\sup_x(R^2+|x|^2)^{-a/2}U\left(x/\sqrt{k},t/k\right),\nonumber\\
&=k^{q}\sup_y\left(R^2+k|y|^2\right)^{-a/2}U(y,\t).
\end{align}
Consider the case $a>a_c$. By Theorem  \ref{Thm:a>ac} and that, for $U_0=k^{-q}u_0$, $L_1=k^{-q}l_1$ and $L_2=k^{-q}l_2$, we have the pointwise estimate
\begin{align*}
\frac{c_1k^{-q}l_1}{4}(1+\t+|y|^2)^{a/2}\leq U(y,\t)\leq\L_0c_2'k^{-q}l_2(\r+\r\t+|y|^2)^{a/2},
\end{align*}
which implies, as $\t\to\I$, that
\begin{align*}
\|u(\cdot,t)\|_{R,a}&\geq k^q\frac{c_1(k^{-q}l_1)}{4}\sup_y\left(\frac{1+\t+|y|^2}{R^2+k|y|^2}\right)^{a/2}\geq\U C'R^{-a}\t^{a/2}=\U Ck^{-\frac{a}{2}}t^{a/2},\\
\|u(\cdot,t)\|_{R,a}&\leq k^q\L_0c_2'(k^{-q}l_2)\sup_y\left(\frac{\r+\r\t+|y|^2}{R^2+k|y|^2}\right)^{a/2}\leq\O C'R^{-a}\t^{a/2}=\O Ck^{-\frac{a}{2}}t^{a/2}.
\end{align*}
Thus we obtain
\begin{align}
\U Ck^{-\frac{a}{2}}t^{a/2}\leq\|u(\cdot,t)\|_{R,a}\leq\O Ck^{-\frac{a}{2}}t^{a/2}\quad(\mbox{as $t\to\I$}).
\end{align}

Therefore, in the case $a>a_c$, the effect of the third order term in (\ref{Eqn:Sobk}) is determined by the factor $k^{-a/2}$ in the asymptotic norm estimates of solutions. The smaller the third order term ($k$ small) the higher the coefficients, whereas the greater the third order term ($k$ large) the smaller the coefficients.\smallskip

At the critical exponent $a=a_c$, the same analysis as above together with Theorem \ref{Thm:a=a_c} shows that
\begin{align}
\U Ck^{-\frac{a_c}{2}}t^{a_c/2}\leq\|u(\cdot,t)\|_{R,a_c}\leq\O C_\V k^{-\frac{a_c}{2}}t^{a_c/2+\V}\quad(\mbox{as $t\to\I$}).
\end{align}
Thus the third order term affects the  asymptotic grow-up of solutions by a factor of $k^{-a_c/2}$.\smallskip

Finally, consider the case $a<a_c$. 
Then we have by Theorem \ref{Thm:a<ac} that
\begin{align*}
\U C\t^{\frac{(\nu+1)_+}{1-p}}(1+\t+|y|^2)^{\frac{a_0}{2}}\leq
U(y,\t)\leq\O C_\V l_2^\V(R^2+|y|^2)^{\frac{a}{2}}\t^{\frac{a_c}{2}+\V}
\end{align*}
and
\begin{align*}
&\|u(\cdot,t)\|_{R,a}\geq k^q\U C\t^{\frac{(\nu+1)_+}{1-p}}\sup_y(R^2+k|y|^2)^{-\frac{a}{2}}(1+\t+|y|^2)^{\frac{a_0}{2}}\geq\U CR^{-a}k^{q-\frac{a_c}{2}}t^{\frac{a_c}{2}},\\
&\|u(\cdot,t)\|_{R,a}\leq k^q\O C_\V(k^{-q}l_2)^\V\sup_{y}\left(\frac{R^2+|y|^2}{R^2+k|y|^2}\right)^{\frac{a}{2}}\t^{\frac{a_c}{2}+\V}\leq\O C_\V l_2^\V k^{q-\frac{a_c}{2}-\V(q+1)}t^{\frac{a_c}{2}+\V},
\end{align*}
as $t\to\I$. So we have
\begin{align}
\U C R^{-a}k^{q-\frac{a_c}{2}}\leq\|u(\cdot,t)\|_{R,a}\leq\O C_{\V'}l_2^\V k^{q-\frac{a_c}{2}-\V'}t^{\frac{a_c}{2}+\V'}.
\end{align}
Thus in this case the third order term affects the asymptotic grow-up of solutions by the factor of $k^{q-a_c/2}$. 

\end{remark}

\section{Appendix}

In our study of sharp asymptotic behavior of solutions to (\ref{Eqn:Main}), Section \ref{Sec:Grow-up}, we shall need the following pointwise estimates involving the Bessel potential operator $\CB$ and the pseudoparabolic green operator $\CG(t)$ acting on the weight-type functions:
\begin{align}
\LL_{\r,d}=(\r+|x|^2)^{d/2},
\end{align}
where $\r>0$ and $d\in\R$.

\begin{lemma}\label{Lem:BesSup}

Let $n\geq1$, $d\in\R$, and $\V,\th\in(0,1)$. There is a constant $\r_0>0$ depending only on $n,d,\V,\theta$ such that if $\r\geq\r_0$ then
\begin{align}
\V(\r+|x|^2)^{\frac{d}{2}}\leq\CB(\r+|x|^2)^{\frac{d}{2}}\leq\theta^{-1}(\r+|x|^2)^{\frac{d}{2}}\quad(\mbox{on $\R^n$}).
\end{align}
Moreover, if $d$ is bounded then $\r_0$ is also bounded.

\end{lemma}

\begin{proof}

Let $\LL_{\r,d}=(\r+|x|^2)^{d/2}=\VS^{d/2}$ where $\VS=\VS(r)=\r+r^2$. Then
\begin{align*}
\begin{cases}
\displaystyle\P_r\LL_{\r,d}=dr\VS^{\frac{d-2}{2}},\\
\vspace{-10pt}\\
\displaystyle\P_{rr}\LL_{\r,d}=d\VS^{\frac{d-2}{2}}+d(d-2)r^2\VS^{\frac{d-4}{2}},\\
\vspace{-10pt}\\
\displaystyle\Lap\LL_{\r,d}=\P_{rr}\LL_{\r,d}+\frac{n-1}{r}\P_r\LL_{\r,d}=\left[nd\VS^{-1}+d(d-2)\left(r^2\VS^{-1}\right)\VS^{-1}\right]\LL_{\r,d}.
\end{cases}
\end{align*}
Thus we have 
\begin{align}\label{Tmp:AppLem1}
\LL_{\r,d}-\Lap\LL_{\r,d}=\left\{1-\left[n+(d-2)\left(r^2\VS^{-1}\right)\right]d\VS^{-1}\right\}\LL_{\r,d}.
\end{align}
Since $r^2\VS^{-1}\in[0,1]$ and $\VS^{-1}\in[0,1/\r]$, we have
\begin{align*}
\left(1-\frac{n|d|+|d(d-2)|}{\r}\right)\LL_{\r,d}\leq \LL_{\r,d}-\Lap\LL_{\r,d} \leq\left(1+\frac{n|d|+|d(d-2)|}{\r}\right)\LL_{\r,d}.
\end{align*}

Now we define
\begin{align*}
\r_1\triangleq\frac{n|d|+|d(d-2)|}{\V^{-1}-1},\quad\r_2\triangleq\frac{n|d|+|d(d-2)|}{1-\th}.
\end{align*}
If $\r\geq\r_1$ then we get $\LL_{\r,d}-\Lap\LL_{\r,d}\leq\V^{-1}\LL_{\r,d}$ hence
\begin{align*}
\CB\LL_{\r,d}\geq\V\LL_{\r,d}.
\end{align*}
If $\r\geq\r_2$ then we get $\LL_{\r,d}-\Lap\LL_{\r,d}\geq\th\LL_{\r,d}$ hence
\begin{align*}
\CB\LL_{\r,d}\leq\th^{-1}\LL_{\r,d}.
\end{align*}
Setting $\r_0=\max\{\r_1,\r_2\}$ the desired two-sided estimate is true for all $\r\geq\r_0$.\QED

\end{proof}

\begin{remark}
\begin{enumerate}
\item[(i)] If $d\in J_n^+$ then $(n+(d-2)(r^2\VS^{-1}))d\geq\min\{dn,d(d+n-2)\}\geq0$ which implies by (\ref{Tmp:AppLem1}) that  $\LL_{\r,d}-\Lap\LL_{\r,d}\leq\LL_{\r,d}$ hence 
\begin{align}
\CB\LL_{\r,d}\geq\LL_{\r,d}.
\end{align}
The latter estimate then implies $\CG(t)\LL_{\r,d}\geq\LL_{\r,d}$.
\item[(ii)] From the lemma, we obtain the two-sided estimate
\begin{align}
e^{-(1-\V)t}\LL_{\r,d}\leq\CG(t)\LL_{\r,d}\leq e^{-(1-\th^{-1})t}\LL_{\r,d},
\end{align}
which is true by the fact that $\CG(t)\vp=e^{-t}\sum_{k=0}^\I\frac{t^k}{k!}\CB^k\vp$.
\end{enumerate}

\end{remark}

\begin{proposition}\label{Prop:TwoSB}

Let $n\geq1$.
\begin{enumerate}
\item[(1)] Let $d\geq0$. There are constants $\r_0,k_0>0$ depending only upon $n,d$ such that if $\r\geq\r_0$ and $k\geq k_0$, then
\begin{align}
\CG(t)(\r+|x|^2)^{\frac{d}{2}}\leq(\r+kt+|x|^2)^{\frac{d}{2}}\quad(x\in\R^n,t\geq0).
\end{align}
Furthermore, if $d$ is bounded then $\r_0$ and $k_0$ are also bounded.
\item[(2)] Let $\r>0$. If either $d\geq0$ with $n\geq2$ or $d\in\{0\}\cup(1,\I)$ with $n=1$, then there is a constant $k_0>0$ depending only upon $n,d,\r$ such that if $0\leq k\leq k_0$ then
\begin{align}
\CG(t)(\r+|x|^2)^{\frac{d}{2}}\geq(\r+kt+|x|^2)^{\frac{d}{2}}\quad(x\in\R^n,t\geq0).
\end{align}
\end{enumerate}

\end{proposition}

\begin{proof}
We can assume $d>0$. Let $u=\CG(t)\left(\r+|x|^2\right)^{d/2}$ and $U=\left(\r+kt+|x|^2\right)^{d/2}$. Then $u$ satisfies the linear pseudoparabolic Cauchy problem
\begin{align}\label{Est:CGL1}
\P_tu-\Lap\P_tu&=\Lap u\quad\mbox{in $Q_\I$},\quad u|_{t=0}=\left(\r+|x|^2\right)^{d/2}.
\end{align}
We denote $\VS=\VS(r,t)=\r+kt+r^2$ so $U=\VS^{d/2}$. Then by direct calculations, we have
\begin{align*}
\begin{cases}
\displaystyle\P_tU=k\left(\frac{d}{2}\right)\VS^{\frac{d-2}{2}},\quad\P_rU=dr\VS^{\frac{d-2}{2}},\quad\P_{rr}U=d\VS^{\frac{d-2}{2}}+d(d-2)r^2\VS^{\frac{d-4}{2}},\\
\vspace{-10pt}\\
\displaystyle\Lap U=\P_{rr}U+\frac{n-1}{r}\P_rU
=\left(2n+2(d-2)(r^2\VS^{-1})\right)\left(\frac{d}{2}\right)\VS^{\frac{d-2}{2}},\\
\vspace{-10pt}\\
\displaystyle\P_t\Lap U=(d-2)\left[n+(d-4)(r^2\VS^{-1})\right](k\VS^{-1})\left(\frac{d}{2}\right)\VS^{\frac{d-2}{2}}.
\end{cases}
\end{align*}
Thus we have
\begin{align*}
\varPi&\triangleq\frac{\Lap U-\P_tU+\P_t\Lap U}{\left(\frac{d}{2}\right)\VS^{\frac{d-2}{2}}},\\
&=2n-k+2(d-2)(r^2\VS^{-1})+(d-2)\left[n+(d-4)(r^2\VS^{-1})\right](k\VS^{-1})
\end{align*}

To prove the first part, we take $\r_0,k_0>0$ sufficiently large so that $U$ is a supersolution to the linear problem. Let
\begin{align*}
\r_0\triangleq2|(d-2)(n+|d-4|)|,\quad k_0=4(n+|d-2|).
\end{align*}
Let $\r\geq\r_0$ and $k\geq k_0$. Since $r^2\VS^{-1}\in[0,1]$ and $\VS^{-1}\in[0,1/\r]$, we have 
\begin{align*}
\VP&\leq 2n-k+|2d-4|+\frac{|(d-2)(n+|d-4|)|}{\r}k,\nonumber\\
&\leq2n-\frac{1}{2}\left(k-4|d-2|\right)\leq0.
\end{align*}
So $U$ is a super-solution of (\ref{Est:CGL1}), that is
\begin{align}
\P_tU-\P_t\Lap U\geq\Lap U\quad\mbox{in $Q_\I$},\quad U|_{t=0}=(\r+|x|^2)^{d/2}.
\end{align}
By the linear pseudoparabolic comparison principle, we obtain
\begin{align*}
\CG(t)(\r+|x|^2)^{d/2}=u\leq U=\left(\r+kt+|x|^2\right)^{d/2},
\end{align*}
which is precisely the desired upper estimate.\smallskip

Now we prove the second part. For $n\geq2$ and $d>0$, it follows that $(d-2)[n+(d-4)(r^2\VS^{-1})]\geq\min\{(d-2)n,(d-2)[n+d-4]\}\geq\min\{(d-2)n,(d-2)(n-2)\}\geq-2n$. This implies
\begin{align*}
\VP&\geq 2n-k+2(d-2)(r^2\VS^{-1})-\frac{2n}{\r}k,\nonumber\\
&\geq 2n-k\left(1+\frac{2n}{\r}\right)+2(d-2)_-.
\end{align*}
Hence $\VP\geq0$ provided 
\begin{align}
0\leq k\leq k_1\triangleq\frac{4+2(d-2)_-}{1+\frac{2n}{\r}}.
\end{align}

For $n=1$ and $d>1$, $(d-2)[1+(d-4)(r^2\VS^{-1})]\geq\min\{d-2,(d-2)(d-3)\}\geq\min\{-1,-\frac{1}{4}\}=-1$. So we have
\begin{align*}
\VP&\geq 2-k+2(d-2)(r^2\VS^{-1})-\frac{k}{\r},\nonumber\\
&\geq2+2(d-2)_--k\left(1+\frac{1}{\r}\right).
\end{align*}
So we have $\VP\geq0$ provided
\begin{align}
0\leq k\leq k_2\triangleq\frac{2+2(d-2)_-}{1+\frac{1}{\r}}.
\end{align}
Taking $k_0\triangleq \min\{k_1,k_2\}$, we have $\VP\geq0$ provided $k\in[0,k_0]$. This implies $U$ is a sub-solution of (\ref{Est:CGL1}), i.e.
\begin{align}
\P_tU-\P_t\Lap U\leq\Lap U\quad\mbox{in $Q_\I$},\quad U|_{t=0}=(\r+|x|^2)^{d/2}.
\end{align}
Therefore 
\begin{align}
\CG(t)(\r+|x|^2)^{d/2}=u\geq U=\left(\r+kt+|x|^2\right)^{d/2},
\end{align}
which is the desired lower estimate.\QED

\end{proof}

\begin{remark}
There is a stronger two-sided estimate, similar to the preceding proposition, for the heat operator. This two-sided estimate can be proved easily using the {\it self-similarity} or {\it scaling property} of the heat equation. See for instance \cite{CazDickEscoWeiss01}. For the pseudoparabolic, there is no scaling transformation owing to the third order viscosity term. 
\end{remark}

\end{document}